\documentclass[a4paper,11pt,oneside]{article}
\usepackage[T1]{fontenc}
\usepackage[utf8]{inputenc}
\usepackage{indentfirst}
\usepackage{fancyhdr}
\usepackage{amssymb}
\usepackage{amsmath}
\usepackage{amsthm}
\usepackage{latexsym}
\usepackage{eucal}
\usepackage{pdfpages}
\usepackage{setspace}
\usepackage{braket}

\usepackage{hyperref}

\theoremstyle{plain}
\newtheorem{theorem}{Theorem}[section]
\newtheorem{lemma}[theorem]{Lemma}
\newtheorem{proposition}[theorem]{Proposition}

\theoremstyle{definition}
\newtheorem{remark}[theorem]{Remark}
\newtheorem{definition}[theorem]{Definition}

\renewcommand{\epsilon}{\varepsilon}
\renewcommand{\div}{\mathrm{div \ }}
\newcommand{\uetau}{u_\epsilon^{(\tau)}}
\newcommand{\petau}{p_\epsilon^{(\tau)}}
\newcommand{\ue}{u_\epsilon}
\newcommand{\we}{w_\epsilon}
\newcommand{\re}{\rho_\epsilon}
\newcommand{\scalar}[2]{\langle #1, #2 \rangle}
\newcommand{\pe}{p_\epsilon}
\newcommand{\dt}{\partial_t}
\newcommand{\dtau}{\partial^{\tau}_t}

\newcommand{\iO}{\int_{\Omega}}
\newcommand{\norma}[2]{\lVert #1 \rVert_{#2}}

\title{Artificial compressibility method for the
	Navier-Stokes-Maxwell-Stefan system}
\author{\large{\bf{Michele Dolce}} \\[1ex]
	\normalsize GSSI - Gran Sasso Science Institute \\
	\normalsize Viale F.\,Crispi, 7 \\
	\normalsize 67100 L'Aquila, Italy.\\
	\normalsize  \href{mailto:michele.dolce@gssi.it}{michele.dolce@gssi.it}
	\and
	\large{\bf{Donatella Donatelli}} \\[1ex]
	\normalsize Department of Information Engineering, Computer Science and Mathematics \\
	\normalsize University of L'Aquila \\
	\normalsize 67100 L'Aquila, Italy. \\
	\normalsize \href{mailto:donatella.donatelli@univaq.it}{donatella.donatelli@univaq.it}
	}
\date{}
\numberwithin{equation}{section}
\begin{document}

\maketitle

\begin{abstract}
The Navier-Stokes-Maxwell-Stefan system describes the dynamics of an incompressible gaseous mixture in isothermal condition. In this paper we set up an artificial compressibility type approximation. In particular we focus on the existence of solution for the approximated system and the convergence to the incompressible case. The existence of the approximating system is proved by means of semidiscretization in time and by estimating the fractional time derivative.
\end{abstract}
\section{Introduction}
The Navier-Stokes-Maxwell-Stefan system is used to describe the dynamics of a multicomponent gaseous mixture, where the velocity of the mixture is described via the classical Navier-Stokes equations while the diffusion of the species, that compose the mixture, is given by the Maxwell-Stefan equations, that describes a non linear cross diffusion. For the time being it is well known that in some situation this type of description is much more realistic than a standard Fick's diffusion, since with the last one, it is not possible to model behaviours that may occur in a multicomponent mixture, see \cite{DUTO}, \cite{WEKR}. One of the possible application of the Maxwell-Stefan equations is for example the case of a patient who has respiratory problem in the lower part of the lung, like asthma, and a mixture of helium and oxygen is used to help him. In this case, the Maxwell-Stefan equations predicts a benefit for the patient while the Fick's law does not. For a detailed discussion on this argument we refer to \cite{BO}. \\ \\ \indent 
The Navier-Stokes-Maxwell-Stefan system is set up in the following way.\\
We consider a system of $N+1$ ideal gases, in isothermal and isolated conditions, with mass densities $\rho_i$, molar mass $M_i$ and velocity $u_i$. The molar masses could be also different.  The total mass density is $\rho^*=\sum_{1}^{N+1}\rho_i$ and the barycentric velocity is $u=(\sum_{i=1}^{N+1}\rho_iu_i)/\rho^*$. \\
Then the continuity equation is given by 
\begin{equation*}
\dt \rho_i+\div (\rho_iu_i)=0.
\end{equation*}
For our purpose it is better to describe the velocity of a particle with respect to the barycentric velocity, so we define the mass fluxes $J_i=\rho_i(u_i-u)$ and the continuity equations can be written as
\begin{equation}
\label{conteq}
\dt \rho_i+\div(J_i+\rho_iu)=0.
\end{equation}
Here we consider the incompressible case, where we assume $\rho^*=\mathrm{const.}$, and for simplicity of notation we set $\rho^*=1$. Summing up \eqref{conteq} for $i=1,\dots, N+1$ we infer that
\begin{equation}
\label{inceq}
\div u=0.
\end{equation}
Using the fact that $\rho^*=1$, the conservation of linear momentum for the barycentric velocity, with external forces $f_i=f$ equal for all the particles, is given by
\begin{equation*}
\dt u+(u\cdot \nabla)u=\div(-p\mathrm{I}+\sigma)+f,
\end{equation*}
where $-p\mathrm{I}+\sigma$ is the tensor that describes the stresses acting on the boundary, $p$ is the usual pressure and $\sigma$ is the viscous stress tensor.  Setting $\sigma=\nu^*(\nabla u+\nabla u^T)$, where for simplicity of notation we consider $\nu^*=1$, we have that 
\begin{equation}
\label{navst}
\dt u+(u \cdot \nabla) u-\Delta u +\nabla p =f.
\end{equation}
To close the above relations we need equations for the velocities $u_i$ or, and this is the case, for the mass fluxes $J_i$. The Maxwell-Stefan equations models a cross-type diffusion that relates the mass fluxes $J_i$ to the densities $\rho_i$ in a nonlinear way. In a famous experiment \textsc{Duncan} and \textsc{Toor}, see \cite{DUTO}, proved that the predictions of the equations fits very well with experimental data. In particular the Maxwell-Stefan equations are very useful to describe the \textit{uphill diffusion} that simplest models, such as Fick's diffusion, can not describe. To write down the equations, we define the molar fractions as 
\begin{equation}
\label{molarfrac}
x_i=\frac{\rho_i/M_i}{\sum_{k=1}^{N+1}\rho_k/M_k}=\frac{\rho_i}{cM_i},
\end{equation}
where $c=\sum_{k=1}^{N+1}\rho_k/M_k$, and in general is not a constant. Notice that by definition it holds that $\sum_{i=1}^{N+1}x_i=1$.\\
In our particular case of ideal, isolated and isothermal conditions with equal forces acting on the mixture, the Maxwell-Stefan equations are given by 
\begin{equation}
\label{maxstef}
\nabla x_i=-\sum_{k=1}^{N+1}\frac{\rho_kJ_i-\rho_iJ_k}{c^2M_iM_kD_{ik}}, \ \ \ \ i=1,\dots,N+1,
\end{equation}
where $D_{ij}$ are the diffusion coefficients.
For a detailed discussion on the thermodynamical derivation of the Maxwell-Stefan equations we refer to \cite{BD}, but them can be also derived from a BGK-type collision operator, see \cite{AAP}.\\ \\ \indent
The Navier-Stokes-Maxwell-Stefan system is given by the following set of equations:
\begin{align}
\label{preeq1}&\dt \rho_i+\div(J_i+\rho_iu)=0, \ \ \ \mathrm{in} \ \Omega, \ t>0\\
\label{preeq2}&\nabla x_i=-\sum_{k=1}^{N+1}\frac{\rho_kJ_i-\rho_iJ_k}{c^2M_iM_kD_{ik}}, \ \ \ \ i=1,\dots,N+1,\\
&\dt u+(u \cdot \nabla) u-\Delta u +\nabla p =f,\\
\label{preeq4}&\div u=0,
\end{align}
where $\Omega\subset\mathbb{R}^d$ $(d\leq3)$ is a bounded domain. 
The initial and boundary conditions are 
\begin{equation}
\label{incond}
\begin{split}
&\rho_{i}(\cdot,0)=\rho_i^0, \ \ \ u(\cdot,0)=u^0 \ \ \ \mathrm{in} \ \Omega,\\
&\nabla \rho_{i }\cdot \nu=0, \ \ \ u=0  \ \ \ \mathrm{on} \ \partial \Omega.
\end{split}
\end{equation}
The system \eqref{preeq1}-\eqref{preeq4} was analysed by Chen, J\"ungel and Stelzer in \cite{CJ} and \cite{JUST}. They have obtained the existence of global weak solution and exponential decay to the stationary state.
The main idea of their paper is to introduce the so called entropy variables, to prove lower and upper bounds for the densities without using maximum principles. \\ \\ \indent
In this paper we will study the system \eqref{preeq1}-\eqref{preeq4} from a different point of view. Since it is well known that the incompressibility constraint \eqref{preeq4} is very expensive for numerical simulation, we construct an approximation which relaxes the condition \eqref{preeq4}. In fact Chorin \cite{CH1}, \cite{CH2}, Témam \cite{TE1}, \cite{TE2} and Oskolkov \cite{OS} introduced the so called \textit{artificial compressibility method} in a bounded domain. 
The idea is to consider the pressure as function of the density and linearize this relation. Then, writing down the continuity equation in terms of the pressure, and considering only the linear part we obtain that 
\begin{equation*}
\epsilon  \dt p+ \div u=0,
\end{equation*}
where $\epsilon$ has the dimension of the Mach number. \\
Notice that since we lose the divergence free condition, we may have some problem such as the increase of the total energy. Therefore, in order to avoid this paradox, we need to change the other equations of the system and we add some corrections terms that in the limit as $\epsilon\to 0$ will vanish.  \\ \\ \indent
The goal of this paper is to show the convergence as $\epsilon\to 0$ of the weak solutions of the artificial compressibility system \eqref{aprdens}-\eqref{aprnav2} towards the weak solutions of the Navier-Stokes-Maxwell-Stefan system \eqref{preeq1}-\eqref{preeq4}.\\
The main stumbling block in the proof of this convergence result is to get the compactness in time for the velocity field. As we will see later on, this will be done by estimating the fractional time derivative of the velocity $u$. It is important to point out that we will deal directly with the different quotient, without introducing a linear interpolation of the discretized solution, as usually done (see for example \cite{TE}). \\ \indent
For completeness we mention that the artificial compressibility approximation was also used in the case of the whole domain \cite{DO1} and exterior domain \cite{DO3} for the Navier-Stokes equation, and modified in a suitable way in the case of the Navier-Stokes-Fourier system \cite{DO2} and the MHD equations \cite{DO4}.
\\ \\ \indent
The paper is organized as follows. In Section \ref{secstate} we first rewrite the system \eqref{preeq1}-\eqref{preeq4} in a more convenient way and then we set up the artificial compressibility approximation system \eqref{aprdens}-\eqref{aprnav2}. Then we state the main results: Theorem \ref{thexistence}, where we prove the global existence for the approximated solution and Theorem \ref{thconv}, where we get the convergence of the approximated system to the original one. \\ \indent 
In Section \ref{mathtools} we recall some mathematical result that we need. In particular in \eqref{defwi} we recover the entropy variables and their properties. \\ \indent
In Section \ref{secproof1} we prove Theorem \ref{thexistence} via semidiscretization in time, that for the approximated Navier-Stokes equations is of interest by itself. \\ \indent
Finally, Section \ref{secproof2} is devoted to the proof of Theorem \ref{thconv}. 
\section{Statement of the problem and main results}
\label{secstate}
The aim of this paper is to define the artificial compressibility approximation for the system \eqref{preeq1}-\eqref{preeq4}.
First of all, we rewrite the equations \eqref{preeq1} and \eqref{preeq2} in order to be handled in a more easier way. Define $\rho=(\rho_1,\dots,\rho_{N+1})^T$ and similarly $J$ and $x$. Then we write \eqref{preeq1} and \eqref{preeq2} as follows 
\begin{align}
\label{matr1}&\dt \rho+(u\cdot \nabla)\rho+\div(J)=0,\\
\label{matr2}&\nabla x=-A(\rho)J,
\end{align}
where the matrix $A(\rho)=(A_{ij})_i^j$, setting $d_{ij}=1/(c^2M_iM_jD_{ij})$, is defined as 
\begin{equation}
\label{matrA}
\begin{split}
A_{ij}&=-d_{ij}\rho_i \ \ \ \mathrm{if} \ i \neq j, \ i,j=1,\dots,N+1, \\
A_{ii}&=\sum_{k=1, k\neq i}^{N+1}d_{ik}\rho_{k} \ \ \ \mathrm{if} \ i=1,\dots,N+1.
\end{split}
\end{equation}
In addition  it is possible to reduce \eqref{matr1} and \eqref{matr2}, to the first $N$ species, using the relation $\rho_{N+1}=1-\sum_{i=1}^{N}\rho_i$. Define $\rho'=(\rho_1,\dots,\rho_N)$ and similarly $x'$ and $J'$. Introduce the following auxiliary matrix 
\begin{equation}
P=I_{N+1}+e_{N+1}\otimes \tilde{e}, \ \ \ P^{-1}=I_{N+1}-e_{N+1}\otimes \tilde{e}
\end{equation}
where $\{e_i\}_{i=1}^{N+1}$ is the canonical basis of $\mathbb{R}^{N+1}$, as column vectors, and we set $\tilde{e}=\sum_{i=1}^{N}e_i$. Now let $v\in \mathbb{R}^{N+1}$ such that $v_{N+1}=1-\sum_{i=1}^{N+1}v_i$, then it holds that $Pv=(v_1,\dots,v_N,0)^T$. So applying on the left $P$, to equations \eqref{matr1} and \eqref{matr2}, using the linearity of the operators and the incompressibility condition, a standard computation shows that \eqref{matr1} and \eqref{matr2} are equivalent to the following,
\begin{align}
\label{eqrho'}&\dt \rho'+(u \cdot \nabla)\rho'+\div J'=0,\\
\label{eqx'}&\nabla x'=-A_0(\rho')J',
\end{align} 
where the matrix $A_0(\rho')=(A_{ij}^0)_i^j$, recalling that $d_{ij}=1/(c^2M_iM_jD_{ij})$, is given by 
\begin{equation}
\label{defA0}
\begin{split}
A^0_{ij}&=-(d_{ij}-d_{i,N+1})\rho_i \ \ \ \mathrm{if} \ i\neq j, i,j=1,\dots,N, \\
A^0_{ii}&=\sum_{k=1, k\neq i}^{N}(d_{ik}-d_{i,N+1})\rho_k+d_{i,N+1} \ \ \ \mathrm{for} \ i=1,\dots,N.
\end{split}
\end{equation}  \\
Hence the Navier-Stokes-Maxwell-Stefan system in the incompressible case, when reduced to the first $N$ species, is the following
\begin{align}
&\dt \rho'+(u\cdot \nabla)\rho' +\div J'=0, \ \ \ \mathrm{in} \ \Omega, \ t>0 \label{inceq1}\\
&\nabla x'=-A_0(\rho')J'\label{inceq2},\\
&\dt u + (u\cdot \nabla )u - \Delta u + \nabla p= f, \label{inceq3}\\
&\div u=0.\label{inceq4}
\end{align} 

Now we are ready to set up the artificial compressibility approximation for \eqref{inceq1}-\eqref{inceq4}. For the Navier-Stokes part \eqref{inceq3}-\eqref{inceq4} we consider the usual approximation, as in \cite{TE}, while in the equation for densities \eqref{inceq1}, to control the growth of the entropy functional, we need to add a new term.\\ \indent
Finally, the artificial compressibility approximation system, written for the first $N$ species, reads as
\begin{align}
\label{aprdens}&\dt \rho'_\epsilon + u_\epsilon \cdot \nabla \rho'_\epsilon+(\div\ue) \rho'_\epsilon+ \div(J'_\epsilon)=0, \ \ \ \mathrm{in} \ \Omega, \ t>0,  \\
\label{aprmax}& \nabla x'_\epsilon = -A_0(\rho_\epsilon')J'_\epsilon ,\\
\label{aprnav1}
&\dt u_\epsilon + u_\epsilon \cdot \nabla u_\epsilon +\frac{1}{2} (\div u_\epsilon) u_\epsilon -\Delta u_\epsilon + \nabla p_\epsilon =f \\
\label{aprnav2}& \epsilon \ \dt p_\epsilon+ \div u_\epsilon=0,
\end{align}
where $\Omega\subset\mathbb{R}^d$ $(d\leq3)$ is a bounded domain with the following initial and boundary conditions
\begin{equation}
\label{incondeps}
\begin{split}
&\rho_{\epsilon,i}(\cdot,0)=\rho_i^0, \ \ \ u_\epsilon(\cdot,0)=u_0 \ \ \ p_\epsilon(\cdot,0)=p_0 \ \ \ \mathrm{in} \ \Omega,\\
&\nabla \rho_{\epsilon,i }\cdot \nu=0, \ \ \ u_\epsilon=0  \ \ \ \mathrm{on} \ \partial \Omega.
\end{split}
\end{equation}

As we will see the two corrections terms $(\div \ue) \re'$ and $ (1/2)(\div \ue)\ue$ will allow us to control the growth of the entropy functional and will avoid the increase of the total energy respectively, both due to the loss of the incompressibility constraint \eqref{inceq4}.\\
In this paper we prove the existence of solutions for the system \eqref{aprdens}-\eqref{aprnav2}, via semidiscretization in time, that for the approximated Navier-Stokes equations as far as we know has some new aspect. Then we prove the convergence, as $\epsilon\to 0$, to the system \eqref{inceq1}-\eqref{inceq4}.\\ \\ \indent 
Before stating our main theorems, we introduce the following spaces.\\
Let $\Omega\subset \mathbb{R}^d$ be a bounded domain with $\partial \Omega \in C^{1,1}$ and let
\begin{equation}
\label{spaces}
\begin{split}
\mathcal{D}(\Omega;\mathbb{R}^n)&=\Set{\varphi\in C^\infty_0(\Omega;\mathbb{R}^n),\ \mathrm{for} \ n\in \mathbb{N}},\\
\mathcal{H}&=\Set{u\in L^2(\Omega,\mathbb{R}^d): \ \div u=0, \ u\cdot \nu|_{\partial \Omega}=0},  \\
\mathcal{V}&=\Set{u\in H^1_0(\Omega, \mathbb{R}^d) : \ \div u=0 },\\ 
\ \mathcal{V}_2&=\mathcal{V}\cap H^2(\Omega, \mathbb{R}^d),  \\
\widetilde{H}^2(\Omega, \mathbb{R}^N)&=\Set{q \in H^2(\Omega, \mathbb{R}^N) : \ \nabla q \cdot \nu |_{\partial \Omega}=0  },\\
H^{\gamma}\big(0,T;H^1_0(\Omega),L^2(\Omega)\big)&=\Set{u: \norma{u}{L^2(0,T;H^1_0(\Omega))}+\norma{\lvert s\rvert^\gamma\hat{u}}{L^2(0,T;L^2(\Omega))}<+\infty},  
\end{split}
\end{equation}
where $\hat{u}$ is the Fourier transform in time defined as
\begin{equation}
\label{Ftransform}
\hat{u}(s)=\int_\mathbb{R}e^{-2\pi i st}u(t)dt.
\end{equation}

\begin{theorem}[Global existence]
	\label{thexistence}
	For any $\epsilon>0$, let $d=1,2,3$, $T>0$, $D_{ij}=D_{ji}>0$, for $i,j=1,...,N+1, \ i\neq j $. Suppose that $f\in L^2(0,T;H^{-1}(\Omega))$, $u_0\in \mathcal{H}$, $p_0\in L^2(\Omega)$ and take $\rho_1^0,...,\rho_{N+1}^0\in L^1(\Omega)$ be nonnegative functions which satisfy $\sum_{1}^{N+1}\rho^0_i=1$ and $h(\rho_0)<+\infty$, where $\rho_0=(\rho_1^0,...,\rho_{N}^0)$ and $h$ is defined in \eqref{defhrho}. Then there exists a global weak solution $(\ue,\pe,\rho_{\epsilon,1},...,\rho_{\epsilon, N+1}))$ to \eqref{aprdens}-\eqref{aprnav2} (in the sense given in \eqref{weakapru}, \eqref{defweakrho}) such that $\rho_{\epsilon,i}\geq0$, $\sum_{1}^{N+1}\rho_{\epsilon,i}=1$ in $\Omega\times (0,T)$, and 
	\begin{equation*}
	\begin{split}
	& \ue \in L^\infty(0,T; L^2(\Omega))\cap H^\gamma(0,T;H^1_0(\Omega),L^2(\Omega)),\  \text{for any  $0<\gamma<1/4$},\\
	&\rho_{\epsilon, i}\in L^2(0,T;H^1(\Omega)), \ \ \dt \rho_{\epsilon, i} \in L^2(0,T;\widetilde{H}^2(\Omega)'), \ \ i=1,...N+1, \\
	&\pe \in L^\infty(0,T; L^2(\Omega)),  \ \ \dt p_{\epsilon} \in L^2(0,T;L^2(\Omega)).
	\end{split}
	\end{equation*}
\end{theorem}
The hypothesis on $h(\rho_0)$ means we are assuming finite entropy initial condition. The Theorem \ref{thexistence} will be proved in Section \ref{secproof1} via semidiscretization in time of the system \eqref{aprdens}-\eqref{aprnav2}. For the densities we will follow the idea of \textsc{Chen}, \textsc{Stelzer} and \textsc{J\"ungel}, \cite{CJ}, \cite{JUST} that is to introduce the entropy variables and prove the existence via fixed point and compactness arguments, without using a maximum principle. \\ \\ \indent
The convergence theorem that justify the approximation is stated as follows. 
\begin{theorem}[Convergence to the original system]
	\label{thconv}
	There exists a sequence $\{u_{\epsilon'}, p_{\epsilon'}, \rho_{\epsilon'}\}$ of weak solution to the problem  \eqref{aprdens}-\eqref{aprnav2}, given by Theorem \ref{thexistence}, such that for $\epsilon'\to 0$ we have that 
	\begin{align*}
	u_{\epsilon'}\to u \ \mathrm{in} \ &L^2(0,T;L^2(\Omega)) \ \mathrm{strongly}, \\
	& L^2(0,T;H^1_0(\Omega)) \ \mathrm{weakly}, \\   
	\nabla p_{\epsilon'}\rightharpoonup \nabla p \ \mathrm{in}  \ & L^2(0,T;H^{-1}(\Omega)) \ \mathrm{weakly},\\ 
	\rho_{\epsilon'} \to \rho \  \mathrm{in} \ & L^2(0,T;L^2(\Omega)) \ \mathrm{strongly},\\
	& L^2(0,T;H^1(\Omega)) \  \mathrm{weakly},
	\end{align*}
	where $(u,\rho)$ is some weak solution to \eqref{inceq1}-\eqref{inceq4}, and $p$ denotes the associated pressure to the Navier-Stokes equations.
\end{theorem} 

The Theorem \ref{thconv} will be a consequence of some uniform in $\epsilon$ estimates that we obtain while proving the existence of the approximating solutions, and we prove it Section \ref{secproof2}. \\ \indent 
Before proving these results, in the next section we recover some mathematical properties that we need in the sequel. 
\section{Mathematical Tools}
\label{mathtools}
We divide the discussion into two subsection, we begin by recalling some general mathematical result and then we recover important mathematical properties specific of the model under consideration.
\subsection{General mathematical result} 
First of all we define an operator and state some property of it, for the detailed discussion we refer to \cite[Chapter III, \S 8]{TE}
\begin{definition} 
	\label{bhat}
	Let $\Omega\subset\mathbb{R}^d$ a bounded domain. Given $u\in H^1_0(\Omega)$ and $v, w \in H^1(\Omega)$, we define the following trilinear form:
	\begin{equation}
	\label{defbhat}
	\hat{b}(u,v,w)=\iO (u \cdot \nabla v) \cdot w \ dz+ \frac{1}{2}\iO(\div u) v\cdot w \ dz.
	\end{equation}
\end{definition}
Then the operator $\hat{b}$ has the following properties.
\begin{lemma}
	\label{propbhat}
	Let $\Omega \subset \mathbb{R}^d$ a bounded domain. Given $u\in H^1_0(\Omega)$ and $v, \ w \in H^1(\Omega)$, then we have the following properties:
	\begin{align}
	\label{bhat1}
	\hat{b}(u,v,w)&=\frac{1}{2}\bigg(\iO (u \cdot \nabla v) \cdot w \ dz-\iO (u \cdot \nabla w) \cdot v \ dz\bigg), \\
	\label{bhat2}
	\hat{b}(u,v,w)&\leq C(\Omega)\lVert u\rVert_{H^1_0(\Omega)} \lVert v \rVert_{H^1(\Omega)} \lVert w \rVert_{H^1(\Omega)}, \\
	\label{bhat3}
	\hat{b}(u,v,w)&=-\hat{b}(u,w,v),\\
	\label{bhat4}
	\hat{b}(u,v,v)&=0.
	\end{align}
\end{lemma}
We recover also a consequence of the Brouwer's fixed point theorem.
\begin{lemma}\label{lemmabr}
	Let $X$ be a finite dimensional Hilbert space, with scalar product $\scalar{\cdot}{\cdot}$ and norm $\norma{\cdot}{}$. Let $P:X\to X$ a continuous map such that 
	\begin{equation*}
	\scalar{P(x)}{x}>0 \ \text{for} \ \norma{x}{}=k>0.
	\end{equation*}
	Then there exist $\bar{x}\in X$ with $\norma{\bar{x}}{}< k$ such that 
	\begin{equation*}
	P(\bar{x})=0.
	\end{equation*}
\end{lemma}
This Lemma will be useful in the proof of existence for the Navier-Stokes part of our system.\\ \indent
The compactness in time for the velocity is proved by the Fourier transform. In particular we will need to control the behaviour at infinity of the Fourier transform of a compactly supported function. In general we have the Paley-Wiener Theorem, but here we state and prove only the particular case that we will use.
\begin{lemma}
	\label{paleywiener}
	Consider $\varphi\in C^\infty_0(\mathbb{R})$ and assume that $\mathrm{supp}\ \varphi \subset [-R,R]$. Then for any $N>0$ it holds that 
	\begin{equation*}
	|\widehat{\varphi}(s)|\leq CR(1+|s|)^{-N}. 
	\end{equation*}
\end{lemma}
\begin{proof}
	By definition of the Fourier transform, we have that for any $N>0$, 
	\begin{equation*}
	(-2\pi is)^N\widehat{\varphi}(s)=\int_{-R}^{R}\frac{d^N}{dt^N}(e^{-2\pi i st})\varphi(t)dt=(-1)^N\int_{-R}^{R}e^{-2\pi i st}\varphi^{(N)}(t)dt,
	\end{equation*}
	where with $\varphi^{(N)}$ we denote de $N$-$th$ derivative. So we have that 
	\begin{equation*}
	|\widehat{\varphi}(s)|\leq C(1+|s|)^{-N}\int_{-R}^{R}|\varphi^{(N)}(t)|dt\leq CR(1+|s|)^{-N},
	\end{equation*}
	where the last inequality is true since $\varphi^{(N)}\in C^\infty_0(\mathbb{R})$.
\end{proof}
Next we state here the version of Leray-Schauder fixed point Theorem that we use for the existence of the densities.
\begin{theorem}[Leray-Schauder fixed point theorem]
	\label{lerayschauder}
	Let $X$ be a Banach space and $T:X\times[0,1]\to X$ a compact map such that
	\begin{itemize}
		\item $T(x,0)=0$ for each $x\in X$,
		\item there exists a constant $M>0$ such that for each pair $(x,\gamma)\in X\times [0,1]$ which satisfies $x=T(x,\gamma)$, we have
		\begin{equation*}
		\lVert x \rVert <M.
		\end{equation*}
	\end{itemize}
	Then $x$ is a fixed point of the map $T_1:X\to X $ given by $T_1y=T(y,1), \ y\in X$.
\end{theorem} 
\bigskip
Now we recall also a classical compactness theorem involving fractional derivatives.
\begin{theorem}
	\label{thHgamma}
	Let $X, \ B$ and $Y$ be Banach spaces such that the embedding $X\hookrightarrow B$ is compact and the embedding $B\hookrightarrow Y$ is continuous. Let $K\subset\mathbb{R}$ be a compact set.\\
	Then for any $\gamma>0$, $H^\gamma(K;X,Y)$ is compactly embedded in $L^2(K;B)$. 
\end{theorem}
For details see \cite[Corollary 5]{SI}, or \cite[Ch.III, \S 2, Theorem 2.2]{TE}.\\\\ \indent
Finally we present a version of the Aubin-Lion's Lemma given by \textsc{M. Dreher} and\textsc{ A. J\"ungel }in \cite[Theorem 1]{DRJ}, that is very useful for the compactness in time for the densities.
The notation will be the following, let $T>0$, $N\in \mathbb{N}$, $\tau= T/N$, and set $t_k=k\tau$, $k=0,\dots,N$. Furthermore, let
\begin{equation*}
(S_hu)(x,t)=u(x,t-h), \ \ \ t\geq h >0,
\end{equation*}
be the shift operator.
\begin{theorem}
	\label{extAubin}
	Let $X, \ B$ and $Y$ be Banach spaces such that the embedding $X\hookrightarrow B$ is compact and the embedding $B\hookrightarrow Y$ is continuous. Furthermore, let either $1\leq p <\infty$, $r=1$ or $p=\infty$, $r>1$, and let $\{v_\tau\}$ be a sequence of functions, which are constant on each subinterval $(t_{k-1},t_k)$, satisfying 
	\begin{equation}
	\frac{1}{\tau}\norma{v_\tau-S_\tau v_\tau}{L^r(\tau,T; Y)}+\norma{v_\tau}{L^p(0,T; X)}\leq C_0 \ \ \ \text{for  all} \ \tau>0,
	\end{equation}
	where $C_0>0$ is a constant which is independent of $\tau$. If $p<\infty$, then $\{v_\tau\}$ is relatively compact in $L^p(0,T;B)$. \\
	If $p=\infty$, there exists a subsequence of $\{v_\tau\}$ which converges in each space $L^q(0,T;B)$, $1\leq q <\infty$, to a limit which belongs to $C^0([0,T];B)$.
\end{theorem}
This theorem is an extension of the Aubin-Lion's particularly useful for proof involving discretization in time, because we do not need estimates on the shifting uniform in $\tau$. It is enough to consider $S_\tau u_\tau$ and not $S_h u_\tau$ as in the Aubin-Lion's Lemma. In particular Theorem \ref{extAubin} avoids to construct explicitly a linear interpolation of $u_\tau$ necessary for the classical Aubin-Lion's Lemma.
\subsection{Mathematical properties of the model}
We recover some useful properties of the system \eqref{aprdens}-\eqref{aprnav2}, for the proofs and a detailed discussion we refer to \cite{CJ} and \cite{JUST}.
\begin{lemma}
	\label{lemmaA0}
	The matrix $A_0$ is invertible and the elements of its inverse $A_0^{-1}$ are uniformly bounded in $\rho_1,\dots,\rho_N\in[0,1]$
\end{lemma}
This Lemma follows by using the Perron-Frobenius theory, see \cite[Lemma 2.1, Lemma 2.3]{JUST}.\\ \\ \indent
Now we define the entropy variables, we start by introducing the entropy density, or Gibbs free energy, that is
\begin{equation}
\label{defhrho}
h(\rho')=c\sum_{i=1}^{N+1}x_i\ln x_i,
\end{equation}
where $x_i$ is defined in \eqref{molarfrac}, $c=\sum_{k=1}^{N+1}\rho_k/M_k$, that is not a constant, and thanks to the relation $\rho_{N+1}=1-\sum_{k=1}^{N}\rho_k$, we interpret the last species as function of the previous one. \\ 
The entropy variables are defined as the derivative of $h$ with respect to the densities $\rho_i$.
\begin{lemma}
	The entropy variables are given by 
	\begin{equation}
	\label{defwi}
	w_i=\frac{\partial h(\rho')}{\partial \rho_i}=\frac{\ln x_i}{M_i}-\frac{\ln x_{N+1}}{M_{N+1}}, \ \ \ i=1,\dots,N.
	\end{equation}
\end{lemma}
The equality \eqref{defwi} follows from a standard computation. \\
It can be proven that we can invert the relation \eqref{defwi}, and also the definition of the molar concentrations, that we recall is $x_i=\rho_i/(cM_i)$, can be inverted. These two properties are given in \cite[Lemma 6, Lemma 7]{CJ}, and by combining these results we get the following.
\begin{lemma}
	\label{wtorho}
	Let $w=(w_1,\dots,w_N)\in \mathbb{R}^N$ be given. Then there exists unique $(\rho_1,\dots,\rho_N)\in (0,1)^N$ satisfying $\sum_{i=1}^{N}\rho_i<1$ such that \eqref{defwi} holds for $\rho_{N+1}=1-\sum_{i=1}^{N}\rho_i$, $x_i=\rho_i/(cM_i)$ and $c=\sum_{i=1}^{N+1}\rho_i/M_i$. In addition the mapping $\mathbb{R}^N \to (0,1)^N$, $\rho'(w)=(\rho_1,\dots,\rho_N)$, is bounded.
\end{lemma}
\begin{remark}
	Observe that we can invert \eqref{defwi} to obtain only positive densities.
\end{remark}\bigskip
A characterization of the Hessian  of $h(\rho')$ with respect to the densities is given by 
\begin{equation}
\label{Hij}
\begin{split}
H_{ij}&:=\frac{\partial^2h(\rho')}{\partial \rho_i\partial \rho_j}=\frac{\partial w_i}{\partial \rho_j}\\
&=\frac{\delta_{ij}}{M_i\rho_i}+\frac{1}{M_{N+1}\rho_{N+1}}-\frac{1}{c}\bigg(\frac{1}{M_i}-\frac{1}{M_{N+1}}\bigg)\bigg(\frac{1}{M_j}-\frac{1}{M_{N+1}}\bigg),
\end{split}
\end{equation}
for $i,j=1,\dots,N$, and where $\delta_{ij}$ is the Kronecker delta. \\
The following Lemma states the convexity of the entropy, for the proof see \cite[Lemma 9]{CJ}.
\begin{lemma}
	\label{lemmaH}
	The matrix $(H_{ij})$, whose elements are defined in \eqref{Hij}, is symmetric and positive definite for all $\rho_1,\dots,\rho_N>0$ satisfying $\sum_{i=1}^{N}\rho_i<1$ 
\end{lemma}

Now we define the following matrix, that represent the derivative of $w_i$ with respect to $x_j$, 
\begin{equation}
\label{Gij}
G_{ij}:=\frac{\partial w_i}{\partial x_j}=\frac{1}{M_{N+1}x_{N+1}}+\frac{\delta_{ij}}{M_ix_i}=c\bigg(\frac{1}{\rho_{N+1}}+\frac{\delta_{ij}}{\rho_i}\bigg), \ \ i,j=1,\dots,N.
\end{equation}  
Then we recall an important Lemma given in \cite[Lemma 10]{CJ}.
\begin{lemma}
	\label{lemmaG}
	It holds for all $\rho_1,\dots, \rho_N>0$ satisfying $\rho_{N+1}=1-\sum_{i=1}^{N}\rho_i>0$:
	\begin{itemize}
		\item[(i)] The matrix $G(\rho')=(G_{ij})$ and its inverse $G^{-1}(\rho')$ are positive definite.
		\item[(ii)] $\nabla w(\rho')=G(\rho')\nabla x'(\rho')$.
		\item[(iii)] The elements of the $N\times N$ matrix $d\rho'/dx'=(\partial\rho_i/\partial x_k)$ are bounded by a constant which depends only on the molar masses $M_i$.
		\item[(iv)]The $N \times N$ matrix $B(\rho')=A_0^{-1}(\rho')G^{-1}(\rho')$ is symmetric, positive definite, and its elements are uniformly bounded.
	\end{itemize}
\end{lemma}
\begin{remark}
	\label{reminversion}
	Notice that thanks to Lemma \ref{lemmaA0}, we can plug the equation \eqref{aprmax} into \eqref{aprdens}, and rewrite it in terms of the entropy variables, namely 
	\begin{equation}
	\label{eqinw}
	\dt \rho'(w_\epsilon) + u_\epsilon \cdot \nabla \rho'(w_\epsilon)+(\div\ue) \rho'(w_\epsilon)- \div(B(w_\epsilon)\nabla w_\epsilon)=0.
	\end{equation}
\end{remark}
We point out that because of [Lemma \ref{lemmaG}, \textit{(iv)}], we have the positiveness of the matrix $B(w)$, while for the matrix $A_0^{-1}(\rho')$ in general we do not have such a strong property. In addition it is possible to be very precise on the positiveness of $B(w)$ in a relevant particular case.
\begin{lemma}
	\label{lemmacoerc}
	Let $w\in H^1(\Omega)$. Then there exists a constant $C_B>0$, only depending on the coefficients $D_{ij}$ and $M_i$ such that 
	\begin{equation}
	\iO \nabla w:B(w)\nabla w \ dz \geq C_B \iO \lvert \nabla \sqrt{x}\rvert^2 \ dz.
	\end{equation}
\end{lemma}
For the proof see \cite[Lemma 12]{CJ}. \\ \\ \indent
Finally we have the computation that justify the introduction of the new term in the equation \eqref{aprdens}.
\begin{lemma}
	\label{lemmafondens}
	Let $w_i$ given as in \eqref{defwi}, $w=(w_1,\dots,w_N)^{T}$ and $\rho'(w)$ obtained by Lemma \ref{wtorho}. Then if $w\in H^1(\Omega)$, for any $u\in H^1_0(\Omega)$ we have that
	\begin{equation}
	\label{hatbplusdiv}
	\hat{b}(u,\rho'(w),w)+\frac{1}{2}\iO\div u \big(\rho'(w)\cdot w\big)dz=-\iO \div u \bigg (\frac{\ln x_{N+1}(\rho'(w))}{M_{N+1}} \bigg)dz,
	\end{equation}
	where $\hat{b}$ is defined in \eqref{defbhat} and $x_{N+1}(\rho'(w))$ is given by Lemma \ref{wtorho}.
\end{lemma}
\begin{remark}
	\label{remextraterm}
	From the Lemma \ref{lemmafondens} it is now more clear how the addition of the term $(\div u)\rho'$, in the equation \eqref{aprdens}, helps to overcome the loss of the divergence free constraint. In fact the left hand side of \eqref{hatbplusdiv} is estimated by a reminder term that, as we will see later on, is bounded in some suitable space.
\end{remark}
\begin{proof}
	By the definition of $\hat{b}$ given in \eqref{defbhat} and the property \eqref{bhat3}, we infer that
	\begin{equation}
	\label{lemmaconto1}
	\begin{split}
	\hat{b}(u,\rho'(w),w)+\frac{1}{2}\iO\div u \big(\rho'(w)\cdot w\big)dz
	=-\iO (u \cdot \nabla w)\cdot \rho'(w) dz.
	\end{split}
	\end{equation}
	Now we handle the right hand side of \eqref{lemmaconto1}. By using \eqref{defwi} and \eqref{molarfrac}, we get
	\begin{equation}
	\begin{split}
	\label{lemmaconto2}
	-\iO (u \cdot \nabla w)\cdot \rho'(w) dz&=-\sum_{i=1}^{N}\iO (u \cdot \nabla w_i) \rho_i(w) dz\\
	&=-\sum_{i=1}^{N} \iO \rho_i(w)u \cdot \bigg(\frac{\nabla x_i}{M_ix_i}-\frac{\nabla x_{N+1}}{M_{N+1}x_{N+1}}\bigg)dz \\
	&=-\sum_{i=1}^{N} \iO cu\cdot \nabla x_i dz +\iO\sum_{i=1}^{N}\rho_i(w)\frac{u\cdot \nabla x_{N+1}}{M_{N+1}x_{N+1}}dz.
	\end{split}
	\end{equation}
	Then, since $\displaystyle{\sum_{i=1}^{N}\rho_i=1-\rho_{N+1}}$ and $\displaystyle{\sum_{i=1}^{N+1}x_i=1}$, we have that
	\begin{equation}
	\label{lemmaconto3}
	\begin{split}
	\iO\sum_{i=1}^{N}\rho_i(w)\frac{u\cdot \nabla x_{N+1}}{M_{N+1}x_{N+1}}dz&=\iO \bigg(\frac{1-\rho_{N+1}(w)}{M_{N+1}x_{N+1}}\bigg)u \cdot \nabla x_{N+1}\\
	&=\frac{1}{M_{N+1}}\iO u\cdot \nabla(\ln x_{N+1})dz-\iO cu \cdot\nabla x_{N+1}dz \\
	&=-\iO \div u \bigg (\frac{\ln x_{N+1}}{M_{N+1}}\bigg ) dz+\sum_{i=1}^{N}\iO cu\cdot \nabla x_i dz.
	\end{split}
	\end{equation}
	Combining \eqref{lemmaconto2} and \eqref{lemmaconto3}, we end up with
	\begin{align}
	\label{lemmaconto4}
	-\iO (u \cdot \nabla w)\cdot \rho'(w) dz=-\iO \div u \bigg (\frac{\ln x_{N+1}}{M_{N+1}}\bigg ) dz.
	\end{align} 
\end{proof}

Now we are ready to prove the Theorems.
\section{Proof of Theorem \ref{thexistence}}
\label{secproof1}
Let us define the notion of weak solution for the system \eqref{aprdens}-\eqref{aprnav2}.
\begin{definition}
	\label{defweakaprinc}
	We say that $(u_\epsilon,p_\epsilon,\rho_\epsilon)$ is a weak solution to \eqref{aprdens}-\eqref{aprnav2} if for any $v \in C^\infty_0(\Omega\times[0,T); \mathbb{R}^d), r\in C^\infty_0(\Omega\times[0,T); \mathbb{R})$ we have that
	\begin{align}
	\notag
	-&\int_{0}^{T}\iO \ue \cdot \dt v \ dzdt+\int_{0}^{T} \hat{b}(\ue,\ue,v) \ dt+\int_{0}^{T}\iO \nabla u_\epsilon:\nabla v \ dzdt \\
	\label{weakapru} &+\int_{0}^{T}\iO\nabla p_\epsilon\cdot v \ dzdt=\iO u_0\cdot v(\cdot,0) \ dz +\int_{0}^{T}\iO f\cdot v \ dzdt, \\ \notag\\
	\label{weakaprp}-\epsilon& \int_{0}^{T}\iO p_\epsilon \dt r \ dz dt +\int_{0}^{T}\iO (\div \ue )r \ dz dt=\iO p_0 w(\cdot,0) \ dz. 
	\end{align}
	While $\rho$ has to satisfy that for any $q \in C^\infty_0(\bar{\Omega}\times[0,T); \mathbb{R}^N)$ with $\nabla q \cdot \nu|_{\partial \Omega}=0$,
	\begin{equation}
	\label{defweakrho}
	\begin{split}
	&-\int_{0}^{T}\int_{\Omega}\rho_\epsilon' \cdot \dt q \ dzdt+\int_{0}^{T}\int_{\Omega}\nabla q : A_0^{-1}(\rho_\epsilon')\nabla x'(\rho_\epsilon') \ dzdt\\
	&+\int_{0}^{T}\hat{b}(u_\epsilon,\rho_\epsilon',q)\ dt+\frac{1}{2}\int_{0}^{T}\iO\div \ue \big(\rho_\epsilon'\cdot q \big)dz\\
	&=\iO(\rho')^0\cdot q(\cdot,0)dz.
	\end{split}
	\end{equation}
	where $(\rho_0)'=(\rho^0_1,...,\rho^0_N)$ and the matrix $A_0$ is defined in \eqref{defA0}. 
\end{definition}
Observe that thanks to Lemma \ref{lemmaA0}, we plug equation \eqref{aprmax} into \eqref{aprdens} and then define the notion of weak solution. \\ \indent
To prove the Theorem \ref{thexistence}, we firstly set up the semidiscretization in time for the whole system and prove the existence of an approximating solution. 
\\ Then we prove a priori estimates and finally we pass to the limit.
\subsection{Approximated problems}
From now on we will use the following notation, let $u, v\in \mathbb{R}^n$ and $A, B \in\mathbb{R}^{n\times n}$, for any $n\in\mathbb{N}$, then
\begin{equation*}
\iO u\cdot v \ dz= \scalar{u}{v}, \ \ \ \iO A:B \ dz=\scalar{A}{B}.
\end{equation*}
Consider $M\in \mathbb{N}$ and set $\tau = T/M$. Let $k=1,\dots,M$.
Define 
\begin{equation}
\label{ftau}
f^{k-1}=\frac{1}{\tau}\int_{(k-1)\tau}^{k\tau}f(t,\cdot)dt, \ \text{for  any} \ k=1,\dots,M.
\end{equation}

Since we want to solve our problem via entropy variables, in order to have some regularity properties we have to define a regularized semidiscretization version of the equation \eqref{eqinw}. For the equation \eqref{aprnav1} and \eqref{aprnav2}, we proceed as for the Navier-Stokes. Hence the semidiscretization for the system \eqref{aprdens}-\eqref{aprnav2} is defined as follows.  \\ \indent

Given $\ue^{k-1}\in L^2(\Omega;\mathbb{R}^d)$, $\pe^{k-1}\in L^2(\Omega;\mathbb{R})$ and $w_\epsilon^{k-1}\in L^\infty(\Omega;\mathbb{R}^N)$, we want to solve the following problems:	for any $v\in H^1_0(\Omega;\mathbb{R}^d)$, $r \in L^2(\Omega;\mathbb{R})$ and $q\in\widetilde{H}^2(\Omega;\mathbb{R}^N)$
\begin{align}
\label{disc1}&\scalar{\frac{\ue^k-\ue^{k-1}}{\tau}}{v}+\hat{b}(\ue^k,\ue^k,v)+\scalar{\nabla \ue^k}{\nabla v}+\scalar{\nabla \pe^k}{v}=\scalar{f^k}{v},\\
\label{disc2}&\epsilon\scalar{\frac{\pe^k-\pe^{k-1}}{\tau}}{r}+\scalar{\div \ue^k}{r}=0, \\ 
\begin{split}
\label{disc3}&\scalar{\frac{\rho'(w_\epsilon^k)-\rho'(w_\epsilon^{k-1})}{\tau}}{q}+\scalar{\nabla q}{B(w_\epsilon^k)\nabla w_\epsilon^k}+\hat{b}(\ue^k,\rho'(w_\epsilon^k),q)\\
&+\frac{1}{2}\scalar{(\div \ue^k)\rho'(w_\epsilon^k) }{q}+\lambda\big[\scalar{\Delta w_\epsilon^k}{\Delta q}+\scalar{w_\epsilon^k}{q}\big]=0.
\end{split}
\end{align}

Our first purpose is to prove the existence of a solution $(\ue^k,\pe^k,\we^k)$ to \eqref{disc1}-\eqref{disc3}, and then define the approximated solution in time. \\ \indent
We will obtain estimates on the velocity field by only using the equations \eqref{disc1} and \eqref{disc2}. To prove the existence of a solution for the equation \eqref{disc3}, it will be sufficient to have an $H^1_0(\Omega)$ bound for the velocity $u_\epsilon^k$. For this reason we start by proving the existence of a solution $(u_\epsilon^k,p_\epsilon^k)$ to equations \eqref{disc1} and \eqref{disc2}.
\subsubsection{Existence of $(u_{\epsilon}^k,\pe^k)$} 
We resume the existence result in the following proposition.
\begin{proposition}
	\label{lemmauekpek}
	Let $\ue^{k-1}\in L^2(\Omega)$ and $\pe^{k-1}\in L^2(\Omega)$. Then there exist a solution $(\ue^k,\pe^k)$ to \eqref{disc1} and \eqref{disc2} such that $\ue^k \in H^1_0(\Omega), \ \pe^k \in L^2(\Omega)$.
\end{proposition}

\begin{proof}
	We use the Galerkin procedure. Consider an orthonormal basis $\{v_i\}$ of $H^1_0(\Omega;\mathbb{R}^d)$, where $v_i\in\mathcal{D}(\Omega;\mathbb{R}^d)$, and an orthonormal basis $\{r_i\}$ of $L^2(\Omega;\mathbb{R})$, with $r_i \in \mathcal{D}(\Omega;\mathbb{R})$, and $\mathcal{D}(\Omega;\cdot)$ is defined in $\eqref{spaces}_1$. 
	\\ 	For each $m\in \mathbb{N}$, we define an approximating solution $u^k_{\epsilon, m}, \ p^k_{\epsilon, m}$ by projection on the first $m$-elements of the basis,
	\begin{equation*}
	\begin{split}
	u^k_{\epsilon, m} &=\sum_{i=1}^{m}\alpha^i_{\epsilon, m}v_i,  \ \ \text{for} \ \alpha^i_{\epsilon, \ m}\in \mathbb{R}\\
	p^k_{\epsilon, m}&=\sum_{j=1}^{m}\beta^j_{\epsilon,m}r_j, \ \ \text{for} \  \beta^j_{\epsilon,m}\in \mathbb{R}.
	\end{split}
	\end{equation*}
	In the following computations, for simplicity of notation, we omit the sub index $\epsilon$.\\ \indent
	For any $v\in \mathrm{span}\{v_1,\dots,v_m\}$ and $r\in \mathrm{span}\{r_1,\dots,r_m\}$, we have to find a solution of
	\begin{align}
	\label{discm1}&\scalar{\frac{u^k_m-u^{k-1}}{\tau}}{v}+\hat{b}(u_m^k,u_m^k,v)+\scalar{\nabla u_m^k}{\nabla v}+\scalar{\nabla p_m^k}{v}=\scalar{f^k}{v},\\
	\label{discm2}&\epsilon\scalar{\frac{p_m^k-p^{k-1}}{\tau}}{r}+\scalar{\div u^k_m}{r}=0,
	\end{align}
	We observe that \eqref{discm1} and \eqref{discm2} are nonlinear equations for the coefficients of $u_m^k$ and $p_m^k$. We will use Lemma \ref{lemmabr} to solve it. Define
	\begin{equation*}
	X=\mathrm{span}\{v_1,\dots,v_m\}, \ \ \ Y=\mathrm{span}\{r_1,\dots,r_m\},
	\end{equation*} then $X\times Y$, equipped with the scalar product induced in the standard way, is a finite dimensional Hilbert space. Then for any $(\phi_m,\psi_m)\in X\times Y$, we define the following map,
	$P:X\times Y\to X\times Y$,  
	\begin{equation*}
	\displaystyle
	P\begin{pmatrix}
	\displaystyle
	\phi_m  \\
	\displaystyle
	\psi_m
	\end{pmatrix} =
	\begin{pmatrix}
	\displaystyle
	\pi_{X}\big(\frac{\phi_m-u^{k-1}}{\tau}+(\phi_m \cdot \nabla)\phi_m + (1/2)(\div \phi_m) \phi_m+\nabla \psi_m-\Delta \phi_m - f\big)\\\displaystyle
	\pi_{Y}\big(\epsilon\frac{\psi_m-p^{k-1}}{\tau}+\div \phi_m\big)\\ 
	\end{pmatrix}, 
	\end{equation*}
	where $\pi_{X}$ and $\pi_{Y}$ are the projections from $H^1_0(\Omega;\mathbb{R}^n)$ and
	$L^2(\Omega;\mathbb{R})$ onto $X$ and $Y$ respectively. $P$ is continuous since it is the composition of continuous operators.
	\\ \indent
	In order to check the hypothesis of Lemma \ref{lemmabr}, let us take $\xi_m=(\phi_m,\psi_m)\in X\times Y$. Thanks to property \eqref{bhat4} of $\hat{b}$ and by integration by parts, we have that
	\begin{align*}
	\scalar{P(\xi_m)}{\xi_m}_{X\times Y}=&\scalar{\frac{\phi_m-u^{k-1}}{\tau}}{\phi_m}+\hat{b}(\phi_m,\phi_m,\phi_m)+\scalar{\nabla \psi_m}{\phi_m}\\
	&+\scalar{\nabla \phi_m}{\nabla \phi_m}-\scalar{f^k}{\phi_m}\\
	&+\epsilon\scalar{\frac{\psi_m-p^{k-1}}{\tau}}{\psi_m}+\scalar{\div \phi_m}{\psi_m}\\
	=&\frac{1}{\tau}\norma{\phi_m}{}^2+\frac{\epsilon}{\tau}\norma{\psi_m}{}^2+\norma{\nabla \phi_m}{}^2\\
	&-\frac{1}{\tau}\scalar{u^{k-1}}{\phi_m}-\frac{\epsilon}{\tau}\scalar{p^{k-1}}{\psi_m}-\scalar{f^k}{\phi_m}.
	\end{align*}
	Then we get that 
	\begin{equation}
	\label{ellipse}
	\scalar{P(\xi_m)}{\xi_m}_{X\times Y}
	\geq \frac{1}{\tau}\norma{\phi_m}{}\big(\norma{\phi_m}{}-\norma{u^{k-1}}{}-\tau\norma{f^{k}}{}\big)+\frac{\epsilon}{\tau}\norma{\psi_m}{}\big(\norma{\psi_m}{}-\norma{p^{k-1}}{}\big).
	\end{equation}
	Now by choosing properly\footnote{To choose properly the constant $C$, observe that if we impose to zero the right hand side of \eqref{ellipse}, for any fixed $\epsilon>0$ we have an equation like $x(x-a)+\epsilon y(y-b)=0$, with $a,b>0$ given. But this equation represent an ellipse, and so the condition in \eqref{straightellipse} is reduced to find a constant $C>0$, such that the straight line $x+y=C$, for $x,y>0$, lies entirely in the exterior of the ellipse, and this is possible since the ellipse is bounded. } a constant $C>0$, that depends only on given quantities, we get that
	\begin{equation}
	\label{straightellipse}
	\scalar{P(\xi)}{\xi}>0 \ \ \text{for} \ \ \norma{\xi}{X\times Y}=\norma{\phi_m}{X}+\norma{\psi_m}{Y}=C.
	\end{equation}
	Then by Lemma \ref{lemmabr} there exist a point $\bar{\xi}=(\bar{\phi},\bar{\psi})$ such that $P(\bar{\xi})=0$. This means that $\bar{\xi}$ is a solution for \eqref{discm1}-\eqref{discm2}, then set $(u_{m}^k,p_{m}^k)=\bar{\xi}$ .\\ \\ \indent
	To perform the limit $m\to \infty$, we need some uniform estimate on $m$ for $(u_m^k,p_m^k)$. We test \eqref{discm1} against $u_m^k$ and \eqref{discm2} against $p_m^k$ and sum up.
	By using \eqref{bhat4} and integrating by parts, we have that 
	\begin{align*}
	\scalar{u^k_m-u^{k-1}}{u_m^k}+\epsilon\scalar{p_m^k-p^{k-1}}{p_m^k}+\tau\norma{\nabla u_m^k}{}^2=\tau\scalar{f^k}{u_m^k},
	\end{align*}
	Then by using the relation 
	\begin{equation*}
	2\scalar{a-b}{a}=\norma{a}{}^2-\norma{b}{}^2+\norma{a-b}{}^2,
	\end{equation*}
	we get that
	\begin{equation}
	\label{estukpk}
	\begin{split}
	\norma{u_m^k}{}^2+\epsilon\norma{p_m^k}{}^2+&2\tau\norma{\nabla u_m^k}{}^2+\norma{u_m^k-u^{k-1}}{}^2+\epsilon\norma{p_m^k-p^{k-1}}{}^2\\ &=\norma{u^{k-1}}{}^2+\epsilon\norma{p^{k-1}}{}^2+2\tau\scalar{f^k}{u_m^k}\\&\leq \norma{u^{k-1}}{}^2+\epsilon\norma{p^{k-1}}{}^2+2\tau C_p\norma{f^k}{} \norma{\nabla u^k_m}{}\\
	&\leq \norma{u^{k-1}}{}^2+\epsilon\norma{p^{k-1}}{}^2+C_p^2\tau\norma{f^k}{}^2+\tau \norma{\nabla u^k_m}{}^2,
	\end{split}
	\end{equation}
	where the last inequalities follows from Poincaré and Young's inequalities. So as $m\to \infty$, we have that up to subsequences, 
	\begin{align*}
	u_{m_j}^k\to u_\epsilon^k \ \ \ &\text{strongly in} \ L^2(\Omega),\\
	&\text{weakly  in} \ H^1_0(\Omega), \\
	p_{m_j}^k\rightharpoonup p_\epsilon^k \ \ \ &\text{weakly  in} \ L^2(\Omega).
	\end{align*}
	Hence $(u_\epsilon^k,p_\epsilon^k)$ is a solution of \eqref{disc1}-\eqref{disc2}.
\end{proof}

\subsubsection{Existence of $w_{\epsilon}^k$}
In this section we prove the existence of a solution for \eqref{disc3}, where $\ue^k$ is the solution to \eqref{disc1} and \eqref{disc2} given by Proposition \ref{lemmauekpek}. \\ 
First of all we notice that due to Lemma \eqref{wtorho}, we can invert the entropy variables to obtain strictly positive densities, for this reason we need to introduce a new space. \\
For $0<\alpha<1$ and define
\begin{equation*}
\begin{split}
Y_{\alpha}=\Bigl\{ & q=(q_1,\dots,q_N)\in L^\infty(\Omega;\mathbb{R}^N): q_i \geq \alpha \ \text{for} \ i=\,\dots,N, \\
& q_{N+1}=1-\sum_{i=1}^{N}q_i\geq \alpha \Bigr\}.
\end{split}
\end{equation*}
First we prove the existence of an approximating solution in the case of strictly positive initial densities, then we will see how to overcome this restriction.
\begin{lemma}
	\label{lemma1aprdens}
	Let $\alpha^{k-1}\in (0,1)$ and $\rho^{k-1}\in Y_{\alpha^{k-1}}$ with $\rho^{k-1}=\rho'(w^{k-1})$. Then there exist $\alpha^k\in (0,1)$ and $\we^k \in \widetilde{H}^2(\Omega;\mathbb{R}^n)$ which solves \eqref{disc3} satisfying $\rho'(\we^k)\in Y_{\alpha^k}$.
\end{lemma}
\begin{remark}
	The result that we obtain has the same structure of the one obtained by \textsc{Chen, X.} and \textsc{J\"ungel, A.} in \cite{CJ}, and we will follow the same outline to prove it. But in our case we have to deal with the $\div \ue$, since we do not have the divergence free condition. In the proof  it will be more clear why it is essential to add the term $(\div \ue) \rho_{\epsilon}'$ in \eqref{aprdens} to bound the entropy variables.
\end{remark}
\begin{proof}
	We have to solve the implicit problem \eqref{disc3}. For this purpose we will set up a fixed point procedure. We define a map $T:X\times [0,1]\to X$, in a suitable Banach space $X$, that associates to the fixed unknowns the solution that we find by means of the Lax-Milgram lemma. Finally we conclude by checking the hypothesis to apply the Leray-Schauder fixed point Theorem \ref{lerayschauder}, for the map $T$. For this part will be important the term that we have introduced to balance the presence of the divergence of the velocity field.\\ \\ \indent
	Let $\bar{w}\in L^\infty(\Omega;\mathbb{R}^N)$ and $\gamma\in [0,1]$. For any $w, \ q\in \widetilde{H}^2(\Omega;\mathbb{R}^n)$, we define 
	\begin{equation}
	\label{a2}
	a_2(w,q)=\lambda\big[\scalar{\Delta w}{\Delta q}+\scalar{w}{q}\big] + \scalar{\nabla q}{B(\bar{w})\nabla w},
	\end{equation}
	\begin{equation}
	\label{F2}
	\begin{split}
	F_2(q)=-&\frac{\gamma}{\tau} \scalar{\rho'(\bar{w})-\rho^{k-1}}{q}\\-& \gamma\bigg[\hat{b}(u_\epsilon^k,\rho'(\bar{w}),q)+\frac{1}{2}\scalar{(\div \ue^k)\rho'(\bar{w})}{q}\bigg].
	\end{split}
	\end{equation}
	We want to prove that there exists a unique $w\in \widetilde{H}^2(\Omega;\mathbb{R}^n)$ which solves    		
	\begin{equation}
	\label{aprdens3}
	a_2(w,q)=F_2(q) \ \ \text{for any} \ q\in \widetilde{H}^2(\Omega;\mathbb{R}^n).
	\end{equation}
	From \eqref{F2}, we deduce that $F_2(\cdot)$ is a linear operator. To prove that it is bounded, first we notice that $\ue^k$ is in $H^1_0(\Omega)$, then we use the boundedness of $\rho'(\bar{w})$ and the continuity of $\hat{b}$, given in Lemma \ref{wtorho} and in \eqref{bhat2} respectively. \\ \indent
	From $(iv)$ in Lemma \ref{lemmaG} we have that the bilinear form $a_2(\cdot,\cdot)$ is bounded and that $B(\bar{w})$ is positive definite. So it follows that   		
	\begin{equation*}
	a_2(w,w) \geq C\lVert w \rVert_{H^2(\Omega)},
	\end{equation*}
	hence $a_2(\cdot,\cdot)$ is a bounded bilinear coercive operator. Then the Lax-Milgram Lemma provides the existence of a unique solution $w$ in $ \widetilde{H}^2(\Omega;\mathbb{R}^n)$ to the equation \eqref{aprdens3}. \\ \\ \indent
	To conclude the proof of Lemma \ref{lemma1aprdens}, we apply the Leray-Schauder fixed point Theorem \ref{lerayschauder}. Define a map $T:L^\infty(\Omega;\mathbb{R}^N)\times[0,1]\to L^\infty(\Omega;\mathbb{R}^N)$, 
	\begin{equation*}
	T(\bar{w},\gamma)=w, \ \text{where} \ w \ \text{solves \eqref{aprdens3}}.
	\end{equation*}
	By construction, $T(\bar{w},0)=0$ for all $\bar{w}\in L^\infty(\Omega;\mathbb{R}^N)$. Recalling that  $H^2(\Omega)$ is compactly embedded in $L^\infty(\Omega)$, we have also that $T$ is compact.\\ 
	To apply Leray-Schauder Theorem \ref{lerayschauder}, it remains to prove that there exists a constant $C>0$ such that $\norma{w}{L^\infty(\Omega)} \leq C$ for all $(w,\gamma)\in L^\infty(\Omega;\mathbb{R}^N)\times[0,1]$ satisfying the fixed point problem $w=T(w,\gamma)$. Such fixed point $w \in  L^\infty(\Omega;\mathbb{R}^N)$ solves \eqref{aprdens3} with $\bar{w}$ replaced by $w$. Take $w\in \widetilde{H}^2(\Omega;\mathbb{R}^N)$ as a test function and from Lemma \ref{lemmafondens} we get
	\begin{equation}
	\label{F2w}
	\begin{split}
	F_2(w)=&-\frac{\gamma}{\tau} \scalar{\rho'(w)-\rho^{k-1}}{w}\\&- \gamma\bigg[\hat{b}(u_\epsilon^k,\rho'(w),w)+\frac{1}{2}\scalar{(\div \ue^k)\rho'(w)}{w}\bigg]\\
	=& -\frac{\gamma}{\tau} \scalar{\rho'(w)-\rho^{k-1}}{w} \\
	& +\frac{\gamma}{M_{N+1}}\scalar{\div \ue^k}{\ln x_{N+1}(\rho'(w))},
	\end{split}
	\end{equation}
	and we point out that here is where we strongly use the new term in equation \eqref{aprdens}. 	Now by using \eqref{F2w}, the relation $a_2(w,w)-F_2(w)=0$ is equivalent to   	    
	\begin{equation}
	\label{aprdensk}
	\begin{split}
	&\frac{\gamma}{\tau} \scalar{\rho'(w)-\rho^{k-1}}{w}+\lambda \lVert w \rVert_{H^2(\Omega)}^2\\ &+\scalar{\nabla w}{B(w)\nabla w} =\frac{\gamma}{M_{N+1}}\scalar{\div \ue^k}{\ln x_{N+1}(\rho'(w))}.
	\end{split}
	\end{equation}
	By Lemma \ref{lemmaH} the entropy density $h$, defined in \eqref{defhrho}, is convex. So we get  
	\begin{equation}
	\label{aprdensh}
	h(\rho'(w))-h(\rho^{k-1})\leq \frac{\partial h}{\partial \rho'}\bigg\rvert_{\rho'(w)} \cdot (\rho'(w)-\rho^{k-1})=w\cdot (\rho'(w)-\rho^{k-1}),
	\end{equation}
	where the last equality follows from \eqref{defwi}.
	\\
	Combining the positiveness of $B(w)$ and \eqref{aprdensh} with \eqref{aprdensk}, we get 
	\begin{equation}
	\begin{split}
	\gamma\iO h(\rho'(w))dz + \lambda \tau \lVert w \rVert_{H^2(\Omega)}^2\leq& \ \gamma \iO h(\rho^{k-1})dz\\&+\gamma \tau\iO \div(\ue^k)\frac{\ln \big[x_{N+1}(\rho'(w))\big]}{M_{N+1}}dz
	\end{split}.
	\end{equation}
	By the definition of $h(\rho'(w))$ given in \eqref{defhrho}, and the fact that $\rho'(w)\in Y_\alpha$, for some $\alpha>0$, we get that $h(\rho'(w))$ is well defined and bounded, therefore we can perform the following estimate,
	\begin{equation}
	\label{winh2}
	\begin{split}
	\lambda \tau \lVert w \rVert_{H^2}^2 \leq & \gamma \bigg \lvert  \iO h(\rho^{k-1})dz \bigg \rvert+\gamma \bigg \lvert\iO h(\rho'(w))dz \bigg \rvert\\
	&\gamma\tau\bigg \lvert\iO \div(\ue^k)\frac{\ln \big[x_{N+1}(\rho'(w))\big]}{M_{N+1}}dz\bigg \rvert\\ &\leq C+C_1\norma{\ue^k}{H^1(\Omega)}\norma{\ln \big[x_{N+1}(\rho'(w))\big]}{L^2(\Omega)}\\
	&\leq C,
	\end{split}
	\end{equation}
	where the last inequality follows because, by Lemma \ref{wtorho}, we have that $0<x_{N+1}(\rho'(w))<1$, and we assume finite initial entropy. By \eqref{winh2} we have the uniform bound for $w$ in $H^2(\Omega)$, hence in $L^\infty(\Omega)$ (observe that this bound is independent from $\epsilon$). 
	\\ \indent
	By Leray-Schauder fixed-point Theorem \ref{lerayschauder}, there exist a fixed point $w \in \widetilde{H}^2(\Omega;\mathbb{R}^n)$ solution of $T(w,1)=w$. Then $\we^k=w$ is a solution of \eqref{disc3}. \\
	To conclude the proof we have to show that $\rho'(\we^k)\in Y_{\alpha_k}$ for a proper $\alpha_k$. By Lemma \ref{wtorho}, we know that given $\we^k$ there exist a unique $(\rho_1(\we^k),\dots,\rho_N(\we^k))\in (0,1)^N$ such that $\rho_{N+1}(\we^k):=1-\sum_{i=1}^{N}\rho_i(\we^k)>0$. So we can define  
	\begin{equation*}
	\alpha_k=\min\limits_{1\leq i \leq N} \mathrm{essinf}_{\Omega}\rho_i(\we^k)>0.
	\end{equation*} 
	Therefore we conclude that $\rho'(\we^k)\in Y_{\alpha^k}$. 
\end{proof}

\subsection{Uniform estimates}
In order to perform the limit $(\lambda,\tau)\to 0$, we need uniform estimates with respect to $(\lambda,\tau)$. \\ 
For the velocity and pressure some bounds will be a direct consequence of the inequalities proved in the previous section. While for the densities it will be more involved. We need also estimates that guarantees compactness in time, since we want strong convergence for the nonlinear terms. For this reason we have to control also the time derivative of our quantities. \\ \indent
We proceed as follows, first we give the estimates regarding the approximated Navier-Stokes equation and then we will control the densities.

\subsubsection{Uniform estimates for the approximated Navier-Stokes equations} 

In the  following Lemma we prove uniform bounds for the velocity and the pressure, that essentially follows from previous computations.
\begin{lemma}
	\label{lemmaest}
	It holds that 
	\begin{align}
	\label{estLinf}	&\sup_{1\leq k\leq M}\big[\norma{\ue^k}{L^2(\Omega)}^2+\epsilon\norma{\pe^k}{L^2(\Omega)}^2\big]\leq C(u_0,p_0,f),\\
	\label{estnablauk}&\tau\sum_{j=1}^{M}\norma{\nabla \ue^j}{L^2(\Omega)}^2\leq C(u_0,p_0,f).
	\end{align}
\end{lemma}
\begin{proof}
	By the lower semicontinuity of the norm, from \eqref{estukpk} we get that 
	\begin{equation}
	\label{ukpk}
	\begin{split}
	\norma{\ue^k}{}^2+\epsilon\norma{\pe^k}{}^2+\tau\norma{\nabla \ue^k}{}^2\leq \norma{\ue^{k-1}}{}^2+\epsilon\norma{\pe^{k-1}}{}^2+C\tau\norma{f^k}{}^2,
	\end{split}
	\end{equation}
	summing up \eqref{ukpk} for $k=1,\dots,s$, we get that
	\begin{equation}
	\label{ukpk2}
	\begin{split}
	\norma{\ue^s}{}^2+\epsilon\norma{\pe^s}{}^2+\tau \sum_{j=1}^{s}\norma{\nabla \ue^j}{}^2\leq \norma{u_{0}}{}^2+\epsilon\norma{p_{0}}{}^2+C\tau\sum_{j=1}^{s}\norma{f^j}{}^2,
	\end{split}
	\end{equation}
	The inequality \eqref{estLinf} follows from \eqref{ukpk2}. Summing up \eqref{ukpk} for $k=1,\dots,M$, we have that
	\begin{equation}
	\label{l2h1uk}
	\tau\sum_{j=1}^{M}\norma{\nabla \ue^j}{}^2\leq C(u_0,p_0,f).
	\end{equation} 
\end{proof}
We point out that since we need to pass into the limit in nonlinear terms for the velocity, we need to prove compactness in time. Hence we proceed by estimating the fractional derivative in time for the velocity, which will allow us to apply Theorem \ref{thHgamma}. Therefore we introduce the following step functions, 
\begin{equation}
\label{utau}
\begin{split}
&\ue^{(\tau)}:[0,T]\times\Omega\to \mathbb{R}^d \\
&\ue^{(\tau)}(t,\cdot)=\ue^{k-1}(\cdot) \ \ \text{for} \ \ (k-1)\tau\leq t <k\tau, \ k=1,\dots,M,
\end{split}
\end{equation}
similarly for $\pe^{(\tau)}$, $f^{(\tau)}$, where we recall the definition of $f^{k-1}$ given in \eqref{ftau}. \\
Then we define an operator that approximates the time derivative via a difference quotient, namely, consider a function $g$, then
\begin{equation}
\label{discder}
\dtau g(t):=\frac{g(t)-g(t-\tau)}{\tau}.
\end{equation}
We rewrite the equations \eqref{disc1} and \eqref{disc2} using our definitions, in the following way 
\begin{align}
\label{newdisc1}&\scalar{\dtau \uetau}{v}+\hat{b}(\uetau,\uetau,v)+\scalar{\nabla \uetau}{\nabla v}+\scalar{\nabla \petau}{v}=\scalar{f^{(\tau)}}{v},\\
\label{newdisc2}&\epsilon\scalar{\dtau \petau}{r}+\scalar{\div \uetau}{r}=0,
\end{align}
for any $t\in [\tau,T]$. In the interval $[0,\tau)$ we have simply the initial data.\\ 
Now Lemma \ref{lemmaest}, in terms of our step functions, reads as  
\begin{align}
\label{spaceutau}&\uetau \in L^{\infty}(0,T;L^2(\Omega))\cap L^2(0,T;H^1_0(\Omega))\\
&\sqrt{\epsilon}\petau \in L^\infty(0,T;L^2(\Omega)),
\end{align}
but being the time interval finite, we also know that $\uetau \in L^{1}(0,T;L^2(\Omega))$ and $\sqrt{\epsilon}\petau \in L^1(0,T;L^2(\Omega))$. Which means that the Fourier transform in time of those functions is well defined (after extending the functions to zero outside the interval $[0,T]$). Now we define the following linear operator 
\begin{align*}
&\phi^{(\tau)}:[0,T]\times H^1_0(\Omega)\to\mathbb{R}\\
&\phi^{(\tau)}(\cdot,v)=\scalar{f^{(\tau)}}{v}-\hat{b}(\uetau,\uetau,v)-\scalar{\nabla \uetau}{\nabla v},
\end{align*} 
and we observe that thanks to the continuity of the operator $\hat{b}$ (see \eqref{bhat2}), we have
\begin{equation*}
\norma{\phi^{(\tau)}(t)}{H^{-1}}\leq \norma{f^{(\tau)}(t)}{H^{-1}}+C\norma{\uetau(t)}{H^1_0}^2+\norma{\uetau(t)}{H^1_0},
\end{equation*}
and so by \eqref{spaceutau}, and hypothesis on the force field $f$, we get that 
\begin{equation}
\label{phil1}
\phi^{(\tau)}\in L^{1}(0,T;H^{-1}(\Omega)).
\end{equation}
By using $\phi^{(\tau)}$ we rewrite again the equations \eqref{newdisc1} and \eqref{newdisc2} in the following way 
\begin{align}
\label{findisc1}&\scalar{\dtau \uetau}{v}+\scalar{\nabla \petau}{v}=\scalar{\phi^{(\tau)}}{v},\\
\label{findisc2}&\epsilon\scalar{\dtau \petau}{r}+\scalar{\div \uetau}{r}=0,
\end{align}
for any $t\in [\tau,T]$.\\ \indent
For any function $g$, the Fourier transform of the discretized time derivative, is given by
\begin{equation}
\label{discderfourier}
\widehat{\dtau g}(s)=\bigg(\frac{1-e^{-2\pi i \tau s}}{\tau}\bigg)\widehat{g}(s).
\end{equation}
\begin{remark}
	Observe that, formally, if in \eqref{discderfourier} we Taylor expand the exponential at the first order, we recover that the Fourier transform of the derivative means multiplication by a factor $2\pi i s$, but in our case is not exactly like that because we are approximating the derivative. Later we will overcome this problem.
\end{remark}
Now we are ready to take the Fourier transform in time for the equation \eqref{findisc1} and \eqref{findisc2}, (extending also $\phi^{(\tau)}$ to zero outside $[0,T]$), and we obtain that
\begin{align}
\label{findisc11}&\bigg(\frac{1-e^{-2\pi i \tau s}}{\tau}\bigg)\scalar{\widehat{\uetau}}{v}+\scalar{\nabla \widehat{\petau}}{v}=\scalar{\widehat{\phi^{(\tau)}}}{v},\\
\label{findisc21}&\epsilon\bigg(\frac{1-e^{-2\pi i \tau s}}{\tau}\bigg)\scalar{ \widehat{\petau}}{r}+\scalar{\div \widehat{\uetau}}{r}=0.
\end{align}
Take now $v=\widehat{\uetau}$ and $r=\widehat{\petau}$ and sum up the equations \eqref{findisc11} and \eqref{findisc21}, to obtain 
\begin{equation}
\bigg(\frac{1-e^{-2\pi i \tau s}}{\tau}\bigg)\bigg(\norma{\widehat{\uetau}(s)}{L^2(\Omega)}^2+\epsilon\norma{\widehat{\petau}(s)}{L^2(\Omega)}^2\bigg)=\scalar{\widehat{\phi^{(\tau)}}}{\widehat{\uetau}}.
\end{equation}
We notice that for  $|s|<1/(2\tau)$,
\begin{equation*}
|\tau s|\leq  \big|1-e^{-2\pi  i  \tau s}\big|,
\end{equation*}
and so we get 
\begin{equation*}
|s|\bigg(\norma{\widehat{\uetau}(s)}{L^2(\Omega)}^2+\epsilon\norma{\widehat{\petau}(s)}{L^2(\Omega)}^2\bigg)\leq |\scalar{\widehat{\phi^{(\tau)}}}{\widehat{\uetau}}|, \ \ \text{for} \ |s|<1/(2\tau).
\end{equation*}
We focus only on the velocity, by \eqref{phil1} and property of the Fourier transform, we know that $\widehat{\phi^{(\tau)}}\in L^{\infty}(\mathbb{R};H^{-1}(\Omega))$, so we have the following inequality
\begin{equation}
\label{boundh1}
|s|\norma{\widehat{\uetau}(s)}{L^2(\Omega)}^2\leq C\norma{\widehat{\uetau}(s)}{H^1_0(\Omega)} \ \ \text{for} \ |s|<1/(2\tau).
\end{equation}
With this estimate we can state the following lemma.
\begin{lemma}
	\label{lemmaboundHgamma}
	It holds that
	\begin{equation}
	\label{boundhgamma}
	\int_\mathbb{R}|s|^{2\gamma}\norma{\widehat{\uetau}(s)}{L^2(\Omega)}^2ds\leq C \ \ \text{for  any} \ \ 0<\gamma<1/4,
	\end{equation}
\end{lemma} 

\begin{remark}
	Lemma $\ref{lemmaest}$ and Lemma \ref{lemmaboundHgamma}, means that \begin{equation*}
	\uetau\in H^\gamma\big(0,T;H^1_0(\Omega),L^2(\Omega)\big),
	\end{equation*} and so we are in the hypothesis of Theorem \ref{thHgamma}. In particular we get compactness, as $\tau \to 0$, in $L^2(0,T;L^2(\Omega))$.
\end{remark}
\begin{proof}
	To prove \eqref{boundhgamma}, first of all we have to provide also a bound when $|s|\geq 1/(2\tau)$, but since our function $\uetau$ is compactly supported step function in time, we can use Lemma \ref{paleywiener}. So we observe that 
	\begin{equation*}
	|s|^{2\gamma}\leq2  \frac{1+|s|}{1+|s|^{1-2\gamma}}, \ \ \text{for any} \ s\in \mathbb{R} \ \text{and for any }\ 0<\gamma<1/4.
	\end{equation*}
	Then get 
	\begin{equation}
	\label{pr.lem.hgamma1}\begin{split}
	\int_\mathbb{R}|s|^{2\gamma}\norma{\widehat{\uetau}(s)}{L^2}^2ds&\leq 2\int_\mathbb{R}\frac{1+|s|}{1+|s|^{1-2\gamma}}\norma{\widehat{\uetau}(s)}{L^2}^2ds\\
	&\leq 2\int_\mathbb{R}\norma{\widehat{\uetau}(s)}{L^2}^2ds+2\int_\mathbb{R}\frac{|s|}{1+|s|^{1-2\gamma}}\norma{\widehat{\uetau}(s)}{L^2}^2ds\\
	&\leq 2\norma{\uetau}{L^2_t(H^1_0)}^2+2\int_\mathbb{R}\frac{|s|}{1+|s|^{1-2\gamma}}\norma{\widehat{\uetau}(s)}{L^2}^2ds,
	\end{split}
	\end{equation}
	where the last inequality follows by Plancharel and the fact that $\uetau \in L^2(0,T;H^1_0(\Omega))$ by \eqref{estnablauk}. So we have only to bound the last term in the right hand side. To do that we split the integral in the following way 
	\begin{equation}
	\label{pr.lemma.hgamma2}
	\begin{split}
	&\int_\mathbb{R}\frac{|s|}{1+|s|^{1-2\gamma}}\norma{\widehat{\uetau}(s)}{L^2}^2ds\\
	=&\int_{|s|<1/(2\tau)}\frac{|s|}{1+|s|^{1-2\gamma}}\norma{\widehat{\uetau}(s)}{L^2}^2ds+\int_{|s|\geq 1/(2\tau)}\frac{|s|}{1+|s|^{1-2\gamma}}\norma{\widehat{\uetau}(s)}{L^2}^2ds\\
	\leq& C\int_{|s|<1/(2\tau)}\frac{\norma{\widehat{\uetau}(s)}{H^1_0}}{1+|s|^{1-2\gamma}}ds+CT\int_{|s|\geq 1/(2\tau)}\frac{|s|}{1+|s|^{1-2\gamma}}(1+|s|)^{-2N}ds,
	\end{split}
	\end{equation}
	where we have used \eqref{boundh1} and Lemma \ref{paleywiener}. Now we notice that choosing $N$ big enough
	\begin{equation*}
	CT\int_{|s|\geq 1/(2\tau)}\frac{|s|}{1+|s|^{1-2\gamma}}(1+|s|)^{-2N}ds\leq \widetilde{C}T,
	\end{equation*}
	so we have to bound just the remaining term, that we control as follows 
	\begin{equation}
	\label{pr.lemma.hgamma3}
	\begin{split}
	\int_{|s|<1/(2\tau)}\frac{\norma{\widehat{\uetau}(s)}{H^1_0}}{1+|s|^{1-2\gamma}}ds&\leq \int_\mathbb{R}\frac{\norma{\widehat{\uetau}(s)}{H^1_0}}{1+|s|^{1-2\gamma}}ds\\
	&\leq\bigg(\int_\mathbb{R}\frac{1}{(1+|s|^{1-2\gamma})^2}ds\bigg)^{1/2}\bigg(\int_\mathbb{R} \norma{\widehat{\uetau}(s)}{H^1_0}^2ds\bigg)^{1/2}\\
	&\leq C\norma{\uetau}{L^2_t(H^1_0)},
	\end{split}
	\end{equation}
	where the last is true since $\gamma<1/4$. Putting together \eqref{pr.lem.hgamma1}, \eqref{pr.lemma.hgamma2} and \eqref{pr.lemma.hgamma3}, we get
	\begin{equation*}
	\int_\mathbb{R}|s|^{2\gamma}\norma{\widehat{\uetau}(s)}{L^2}^2ds\leq C\norma{\uetau}{L^2_t(H^1_0)}(1+\norma{\uetau}{L^2_t(H^1_0)})\leq\widetilde{C},
	\end{equation*}
\end{proof}
Finally by \eqref{newdisc2} and the $L^2(0,T;H^1_0(\Omega))$ bound for $\uetau$, see \eqref{spaceutau}, we have that 
\begin{equation}
\label{derpressure}
\norma{\epsilon\dtau \petau}{L^2(0,T;L^2(\Omega))}\leq C,
\end{equation}
which gives us sufficient conditions to pass into the limit in the linear equation \eqref{newdisc2}.
\subsubsection{Uniform estimates for the densities}
In order to prove the estimates for the densities, let us overcome the problem of a general initial value introducing another parameter, that in the end we can send to zero.\\ 
Let $\rho_0=(\rho^0_1,\dots,\rho^0_{N+1})$ satisfying $\rho^0_i\geq0$ for $i=1,\dots,N+1$ with the condition $\sum_{i=1}^{N+1}\rho^0_i=1$. Let be $0<\alpha^0<1/(2(N+1))$ and define
\begin{equation}
\label{rhoeta0}
\rho_i^{\alpha^0}=\frac{\rho^0_i+2\alpha^0}{1+2\alpha^0(N+1)}, \ i=1,\dots,N+1,
\end{equation}
one can check that $\rho^{\alpha^0} \in Y_{\alpha^0}$. \\
Now let be $w^0\in L^\infty(\Omega;\mathbb{R}^N)$ defined by \eqref{defwi}. By iterating the application of Lemma \ref{lemma1aprdens} and starting from $\rho^{\alpha^0}$, we obtain a sequence of approximate solutions $\we^k\in \widetilde{H}^2(\Omega;\mathbb{R}^N)$ to \eqref{disc3} such that $\rho'(\we^k)\in Y_{\alpha^k}$ where $\alpha^k \in (0,1)$. \\ \indent 
In the following, we set $\re^k=\rho'(\we^k)$ for $k\geq 0$. 
In the next Lemma we give some uniform estimates in the time integral of the discretized solutions.
\begin{lemma}
	\label{lemmawh2}
	For any $1\leq k\leq M$ and sufficiently small $\alpha^0>0$, it holds that
	\begin{equation}
	\label{disugaprdens}
	\begin{split}
	&\iO h(\re^k)dz+C_B\tau \sum_{j=1}^{k}\lVert \nabla \sqrt{x(\re^j)}\rVert^2_{L^2(\Omega)}+\lambda \tau\sum_{j=1}^{k}\lVert \we^j \rVert ^2_{H^2(\Omega)} \\
	&\leq \iO h(\rho_0) dz + 1 + C(u_0,p_0,f)T,
	\end{split}
	\end{equation}
	where $\sqrt{x(\re^j)}=(\sqrt{x_1(\re^j)},\dots,\sqrt{x_{N+1}(\re^j)}), \ x_i(\re^j)=\rho^j_{\epsilon,i}/(cM_i)$ for $i=1,\dots,N+1$,
	$c=\sum_{k=1}^{N+1}\rho^j_{\epsilon,k}/M_k$, $C_B$ is obtained from Lemma \ref{lemmacoerc}.
\end{lemma}
\begin{proof}
	Combining \eqref{aprdensk}, \eqref{aprdensh} and using Lemma \ref{lemmacoerc}, after summation over $j=1,\dots,k$ it follows that
	\begin{equation}
	\label{esthxw}
	\begin{split}
	\iO h(\rho^k)dz&+C_B\tau \sum_{j=1}^{k}\lVert \nabla \sqrt{x(\re^j)}\rVert^2_{L^2(\Omega)}+\lambda \tau\sum_{j=1}^{k}\lVert w^j \rVert ^2_{H^2(\Omega)}\\
	&\leq  \iO h(\rho^{\alpha^0}) dz+\tau\sum_{j=1}^{k}C_1\norma{\ue}{H^1_0(\Omega)}\norma{\ln \big[x_{N+1}(\re^j)\big]}{L^2(\Omega)}\\
	&\leq \iO h(\rho^{\alpha^0}) dz +C(u_0,p_0,f)T.
	\end{split}
	\end{equation}
	In the bound \eqref{esthxw} there is still a dependence on $\alpha_0$, but by \eqref{defhrho} $|h|$ is bounded and $h$ is a continuous function. Therefore, recalling that we are in a bounded domain, by dominated convergence theorem, we have
	\begin{equation*}
	\lim\limits_{\alpha^0 \to 0}\iO h(\rho^{\alpha^0})dz=\iO h(\rho_0)dz,
	\end{equation*}
	hence for sufficiently small $\alpha^0>0$, 
	\begin{equation}
	\label{esthrhoeta0}
	\iO h(\rho^{\alpha^0})dz\leq \iO h(\rho_0)dz+1.
	\end{equation}
	By combining \eqref{esthxw} and \eqref{esthrhoeta0}, we get that 
	\begin{equation*}
	\begin{split}
	&\iO h(\rho^k)dz+C_B\tau \sum_{j=1}^{k}\lVert \nabla \sqrt{x(\re^j)}\rVert^2_{L^2(\Omega)}+\lambda \tau\sum_{j=1}^{k}\lVert w^j \rVert ^2_{H^2(\Omega)}\\
	&\leq \iO h(\rho_0)dz+1 +C(u_0,p_0,f)T.
	\end{split}
	\end{equation*}
\end{proof}

\begin{remark}
	It is important to point out that despite the lack of the incompressible constraint, (see \cite[Lemma 14]{CJ}), we can still recover an $L^2$ bound in time for the gradient of the discretized densities and molar concentrations.
\end{remark}
\begin{lemma}
	\label{lemmaderrho}
	It holds that 
	\begin{equation*}
	\tau \sum_{k=1}^{M}\lVert\nabla x(\re^k)\rVert^2_{L^2(\Omega)}+\tau \sum_{k=1}^{M}\lVert\nabla\re^k\rVert^2_{L^2(\Omega)}\leq C(u_0,p_0,\rho_0,f).
	\end{equation*}
\end{lemma}
For the proof we refer to \cite[Lemma 15]{CJ}, just notice that now we have that the constant depends on $p_0$.\\  \\ \indent 
Since we have nonlinearities also in the equation \eqref{aprdens}, we need an estimates on the discrete time derivative of the densities, in order to employ the compactness result given by \textsc{Dreher} and \textsc{J\"ungel} in \cite{DRJ}, see Theorem \ref{extAubin}.
\begin{lemma}
	\label{lemmaforaubin}
	It holds that
	\begin{align}
	\tau \label{estH-2}\sum_{k=1}^{M}\bigg\lVert\frac{\re^k-\re^{k-1}}{\tau}\bigg\rVert_{\widetilde{H}^2(\Omega)'}^2\leq C(u_0,\rho_0,p_0,f).
	\end{align}
\end{lemma}
\begin{proof}
	The proof will follows the same lines of argument as in \cite[Lemma $15$]{CJ}, but again we point out the importance of the additional term $(\div \ue)\re'$ introduced in the equation \eqref{aprdens}, compared to the original system.\\ \indent
	To prove \eqref{estH-2}, first of all we observe that, thanks to H\"older inequality, Sobolev embeddings and by using \eqref{estLinf}, we have
	\begin{equation}
	\label{holdsobcom}
	\begin{split}
	\hat{b}(\ue^k,\re^k,q)&+\frac{1}{2}\iO\div \ue^k (\re^k\cdot q)dz\\&\leq \big(C\norma{\ue}{H^1_0(\Omega)}\norma{\nabla \re^k}{L^2(\Omega)}+C\norma{\ue}{H^1_0(\Omega)}\norma{ \re^k}{L^2(\Omega)}\big)\norma{q}{H^1(\Omega)} \\
	&\leq C_0(u_0,p_0,f)\big(\norma{\nabla \re^k}{L^2(\Omega)}+\norma{ \re^k}{L^2(\Omega)} \big)\norma{q}{H^1(\Omega)}.
	\end{split}
	\end{equation}
	In addition by Lemma \ref{lemmaA0} we have the boundedness of the elements of $A_0^{-1}(\re^k)$. Then from \eqref{disc3}, we deduce that for $q\in\widetilde{H}^2(\Omega)$, we have
	\begin{equation*}
	\begin{split}
	\big| \scalar{\frac{\re^k-\re^{k-1}}{\tau}}{q} \big| \leq& \norma{A_0^{-1}}{L^\infty(\Omega)}\norma{\nabla x(\re^k)}{L^2(\Omega)}\norma{q}{H^2(\Omega)}\\
	&+C_0(u_0,p_0,f)\big(\norma{\nabla \re^k}{L^2(\Omega)}+\norma{ \re^k}{L^2(\Omega)} \big)\norma{q}{H^1(\Omega)}\\
	&+\lambda\norma{w_\epsilon^k}{H^2(\Omega)}\norma{q}{H^2(\Omega)},
	\end{split}
	\end{equation*}
	so we get that
	\begin{equation*}
	\begin{split}
	\big| \scalar{\frac{\re^k-\re^{k-1}}{\tau}}{q} \big| \leq C\big(&\norma{\nabla x(\re^{k})}{L^2(\Omega)}+\norma{\nabla \re^k}{L^2(\Omega)}\\&+\norma{ \re^k}{L^2(\Omega)}+\lambda\norma{w_\epsilon^k}{H^2(\Omega)} \big)\norma{q}{H^2(\Omega)}
	\end{split},
	\end{equation*}
	where $C$ is a constant that depends only on $u_0,\ p_0, \ f$ and the elements of $A^{-1}_0$. Taking the square on both sides, we get that
	\begin{equation}
	\label{estrhok}
	\bigg\lVert\frac{\re^k-\re^{k-1}}{\tau}\bigg\rVert_{\widetilde{H}^2(\Omega)'}^2\leq 4C^2\big(\norma{\nabla x(\re^{k})}{L^2(\Omega)}^2+\norma{\nabla \re^k}{L^2(\Omega)}^2+\norma{ \re^k}{L^2(\Omega)}^2+\lambda\norma{w_\epsilon^k}{H^2(\Omega)}^2 \big).
	\end{equation}
	By Lemma \ref{lemmaderrho}, we observe that we have
	\begin{equation*}
	\begin{split}
	\tau\sum_{k=1}^{M}\big(\norma{\nabla \re^k}{L^2(\Omega)}^2+\norma{ \re^k}{L^2(\Omega)}^2\big) &\leq C(u_0,p_0,\rho_0,f)+\tau\sum_{k=1}^{M}\norma{ \re^k}{L^2(\Omega)}^2\\
	&\leq C(u_0,p_0,\rho_0,f)+CT.
	\end{split}
	\end{equation*}
	Multiplying by $\tau$ and summing up \eqref{estrhok} for $k=1,\dots,M$, we have that
	\begin{equation}
	\begin{split}
	&\tau \sum_{k=1}^{M}\bigg\lVert\frac{\re^k-\re^{k-1}}{\tau}\bigg\rVert_{\widetilde{H}^2(\Omega)'}^2 \\
	&\leq 4C^2\tau\sum_{k=1}^{M}\big(\norma{\nabla x(\re^{k})}{L^2(\Omega)}^2+\norma{ \re^k}{H^1(\Omega)}^2+\lambda^2\norma{w_\epsilon^k}{H^2(\Omega)}^2 \big)\\
	&\leq C(u_0,p_0,\rho_0,f),
	\end{split}
	\end{equation}
	where we have used Lemma \ref{lemmaderrho}, and \eqref{disugaprdens} for the estimate on $\lambda w_\epsilon^k$. 
\end{proof}
With the uniform estimates in $\tau$ and $\lambda$ that we got in this section, we are ready to pass into the limit and to get our solutions.
\subsection{Passage to the limit}
We start with the definition of the approximations of the solutions that we are looking for. We recall the definition of the piecewise constant function given in \eqref{utau},
\begin{equation*}
\begin{split}
&\ue^{(\tau)}:[0,T]\times\Omega\to \mathbb{R}^d \\
&\ue^{(\tau)}(t,\cdot)=\ue^{k-1}(\cdot) \ \ \text{for} \ \ (k-1)\tau\leq t <k\tau, \ k=1,\dots,M,
\end{split}
\end{equation*}
and in the same way it is defined for $p_\epsilon^{(\tau)},  f^{(\tau)},  \re^{(\tau)},  w_{\epsilon}^{(\tau)}$. \\ 
Recall also that the difference quotient, defined in \eqref{discder}, is given by  
\begin{equation*}
\dt^{\tau} \re^{(\tau)}(t)=\frac{\re^{(\tau)}(t)-\re^{(\tau)}(t-\tau)}{\tau}\ \ \ \text{for} \ t\in[\tau,T], 
\end{equation*}
and similarly for $\dt^{\tau} \pe^{(\tau)}$. \\ \indent
From the Lemmas \ref{lemmaest},\ref{lemmaboundHgamma}, \ref{lemmaforaubin} and from \eqref{derpressure}, we get the following uniform estimates:
\begin{align}
\label{estutau}
&\norma{\ue^{(\tau)}}{L^\infty(0,T;H^1_0(\Omega))}+\norma{\ue^{(\tau)}}{H^\gamma(0,T;H^1_0(\Omega),L^2(\Omega))}\leq C, \\
\label{estptau}
&\sqrt{\epsilon}\norma{\pe^{(\tau)}}{L^\infty(0,T;L^2(\Omega))}+\norma{\epsilon\dtau\pe^{(\tau)}}{L^2(0,T;L^2(\Omega))}\leq C, \\
\label{boundx}
&\norma{x(\re^{(\tau)})}{L^\infty(0,T;L^\infty(\Omega))}+\norma{x(\re^{(\tau)})}{L^\infty(0,T;H^1(\Omega))}\leq C,\\
\label{boundrho}
&\norma{\re^{(\tau)}}{L^\infty(0,T;L^\infty(\Omega))}+\norma{\re^{(\tau)}}{L^2(0,T;H^1(\Omega))}+\norma{\dt^\tau\re^{(\tau)}}{L^\infty(0,T;\widetilde{H}^2(\Omega)')}\leq C,
\\
\label{boundw}
&\sqrt{\lambda}\norma{w_\epsilon^{(\tau)}}{L^2(0,T;H^2(\Omega))}\leq C,
\end{align}
where all the constants that appear on the right hand side, depends only on the boundedness of the initial datum and the domain. We want to underline that in the constants is important only that $h(\rho_0)<+\infty$, see \eqref{disugaprdens}, and it does not matter if some initial density is zero.\\ \indent
As a consequence of \eqref{estutau}, \eqref{estptau} and \eqref{boundrho}, and Theorem \ref{thHgamma}, up to subsequences, as $\tau\to0$ we obtain 
\begin{align}
\label{conv1}
\ue^{(\tau)}\rightharpoonup u_{\epsilon} \ \ \ &\text{weakly  in} \ L^2(0,T;H^1_0(\Omega)),\\
\label{conv2}
\ue^{(\tau)}\to \ue \ \ \ &\text{strongly in} \ L^2(0,T;L^2(\Omega)),\\
\label{conv3}
\pe^{(\tau)}\overset{\star}{\rightharpoonup} p_\epsilon \ \ \ &\text{weakly  star  in} \ L^\infty(0,T;L^2(\Omega)),\\
\re^{(\tau)}\rightharpoonup \re' \ \ \ &\text{weakly  in} \ L^2(0,T;H^1(\Omega)).
\end{align}
In addition \eqref{boundrho}, using Theorem \ref{extAubin}, we have that as $\tau \to 0$
\begin{align}
\label{conv4}
\re^{(\tau)}\to \re' \ \ \ \text{strongly  in} \ L^2(0,T;L^2(\Omega)).
\end{align}
Moreover the strong convergence of $\{\re^{(\tau)}\}$ and the boundedness of the elements of $A_0^{-1}$ and $x'$ imply 
\begin{align*}
A_0^{-1}(\re^{(\tau)})\to A_0^{-1}(\re')\ \ \ &\text{strongly  in} \ L^p(0,T;L^p(\Omega)),\\
x(\rho^{(\tau)})\to x(\re') \ \ \ &\text{strongly  in} \ L^p(0,T;L^p(\Omega)),
\end{align*}
for any $p<\infty$. \\
For the sequence $\{\nabla x'(\re^ {(\tau)})\}$, thanks to \eqref{boundx}, we have at least weak convergence, so we infer that
\begin{equation*}
\nabla x(\re^ {(\tau)})\rightharpoonup\nabla x(\re') \ \ \text{weakly  in} \ L^2(0,T;L^2(\Omega)).
\end{equation*}
Finally, thanks to \eqref{boundw}, we notice that as $(\lambda,\tau)\to 0$
\begin{equation}
\lambda w_\epsilon^{(\tau)} \to 0 \ \ \ \text{strongly in} \ L^2(0,T;L^2(\Omega)).
\end{equation}
Now we rewrite the weak formulations of the Definition \ref{defweakaprinc} for our piecewise constant functions. 
For any $v \in C^\infty_0(\Omega\times[0,T); \mathbb{R}^d), \ r\in C^\infty_0(\Omega\times[0,T); \mathbb{R})$ and $q \in C^\infty_0(\bar{\Omega}\times[0,T); \mathbb{R}^N)$ with $\nabla q \cdot \nu|_{\partial \Omega}=0$, we have that
\begin{equation*}
\begin{split}
&\int_{0}^{T}\iO \dt^{\tau} \ue^{(\tau)} \cdot v \ dzdt+\int_{0}^{T} \hat{b}(\ue^{(\tau)},\ue^{(\tau)},v) \ dt+\int_{0}^{T}\iO \nabla u_\epsilon^{(\tau)}:\nabla v \ dzdt \\
&+\int_{0}^{T}\iO\nabla p_\epsilon^{(\tau)}\cdot v \ dzdt=\int_{0}^{T}\iO f^{(\tau)}\cdot v \ dzdt, \\ \\
\epsilon& \int_{0}^{T}\iO \dt^{\tau}p_\epsilon^{(\tau)}   r \ dz dt +\int_{0}^{T}\iO (\div \ue^{(\tau)}) r \ dz dt=0, \\ \\ 
&\int_{0}^{T}\int_{\Omega}\dt^{\tau}\rho_\epsilon^{(\tau)} \cdot  q \ dzdt+\int_{0}^{T}\int_{\Omega}\nabla q : A_0^{-1}(\rho_\epsilon^{(\tau)})\nabla x(\rho_\epsilon^{(\tau)}) \ dzdt\\
&+\int_{0}^{T}\hat{b}(u_\epsilon^{(\tau)},\rho_\epsilon^{(\tau)},q)\ dt+\frac{1}{2}\int_{0}^{T}\iO\div \ue^{(\tau)} \big(\rho_\epsilon^{(\tau)}\cdot q \big)dz\\
&=-\lambda\int_{0}^{T}\iO(\Delta w_\epsilon^{(\tau)}\cdot \Delta q +w_\epsilon^{(\tau)}\cdot q)dzdt.
\end{split}
\end{equation*}
Before passing to the limit, observe that by a change of variable and the compact support of the test functions, we have that
\begin{align*}
\int_{0}^{T}\iO \dt^\tau \ue^{(\tau)}\cdot vdt dz=&-\int_{0}^{T}\iO\ue^{(\tau)}(t)\cdot\dt^{-\tau} v(t) dtdz\\
&-\iO u_0 \cdot \bigg[\frac{1}{\tau}\int_{0}^{\tau}v(t)dt\bigg]dz,
\end{align*}
the same holds for $\pe^{(\tau)}, \re^{(\tau)}$.\\
Observe that in the terms with $\hat{b}$, before taking the limit, we have to use the property \eqref{bhat3}, that is crucial, because otherwise the strong convergence result that we have would not be sufficient to pass into the limit.\\ 
Thanks to the strong and weak convergence results \eqref{conv1}-\eqref{conv4} we obtain
\begin{align*}
-&\int_{0}^{T}\iO \ue \cdot \dt v \ dzdt+\int_{0}^{T} \hat{b}(\ue,\ue,v) \ dt+\int_{0}^{T}\iO \nabla u_\epsilon:\nabla v \ dzdt \\
&+\int_{0}^{T}\iO\nabla p_\epsilon\cdot v \ dzdt=\iO u_0\cdot v(\cdot,0) \ dz +\int_{0}^{T}\iO f\cdot v \ dzdt, \\ \\
-\epsilon& \int_{0}^{T}\iO p_\epsilon \dt r \ dz dt +\int_{0}^{T}\iO (\div \ue )r \ dz dt=\iO p_0 w(\cdot,0) \ dz,\\ \\ 
&-\int_{0}^{T}\int_{\Omega}\rho_\epsilon' \cdot \dt q \ dzdt+\int_{0}^{T}\int_{\Omega}\nabla q : A_0^{-1}(\rho_\epsilon')\nabla x'(\rho_\epsilon') \ dzdt\\
&+\int_{0}^{T}\hat{b}(u_\epsilon,\rho_\epsilon',q)\ dt+\frac{1}{2}\int_{0}^{T}\iO\div \ue \big(\rho_\epsilon'\cdot q \big)dz\\
&=\iO(\rho')^0\cdot q(\cdot,0)dz.
\end{align*}
Then according to the Definition \eqref{defweakaprinc}, $(\ue,\pe,\re)$ is a weak solution to \eqref{aprdens}-\eqref{aprnav2}.\\ 
Finally, notice that all the estimates \eqref{estutau}-\eqref{boundw} are independent from $\alpha_0$, then thanks to the finite entropy assumption, $h(\rho_0)<+\infty$, we can perform the limit as $\alpha_0\to 0$ and conclude the result for a general initial data.
\section{Proof of the Theorem \ref{thconv}}
\label{secproof2}
Notice that all the estimates given in \eqref{estutau}-\eqref{boundrho} are independent from $\epsilon$. Then, by the lower semicontinuity of the norm, we have the same type of bounds for the solution $(\ue, \pe, \re)$ and for $x(\re)$. We resume here the estimates that we have 
\begin{align}
\label{estue}
&\norma{\ue}{L^\infty(0,T;H^1_0(\Omega))}+\norma{\ue}{H^\gamma(0,T;H^1_0(\Omega),L^2(\Omega))}\leq C, \\
\label{estpe}
&\sqrt{\epsilon}\norma{\pe}{L^\infty(0,T;L^2(\Omega))}+\norma{\epsilon\dt\pe}{L^\infty(0,T;L^2(\Omega))}\leq C, \\
\label{estxre}
&\norma{x(\re')}{L^\infty(0,T;L^\infty(\Omega))}+\norma{x(\re')}{L^\infty(0,T;H^1(\Omega))}\leq C,\\
\label{estre}
&\norma{\re'}{L^\infty(0,T;L^\infty(\Omega))}+\norma{\re'}{L^2(0,T;H^1(\Omega))}+\norma{\dt\re'}{L^\infty(0,T;\widetilde{H}^2(\Omega)')}\leq C.
\end{align}
From \eqref{estue}-\eqref{estre}, we infer that, up to subsequence, 
\begin{equation}
\label{weak1}
\begin{split}
u_{\epsilon'}\rightharpoonup u \   \text{in}  \ & L^2(0,T;H^1_0(\Omega))  \ \text{weakly}, \\ & L^\infty(0,T;L^2(\Omega)) \  \text{weakly  star},
\end{split}
\end{equation}
\begin{equation}
\label{weakpres}
\sqrt{\epsilon'}p_{\epsilon'}\overset{\star}{\rightharpoonup} \chi  \  \text{in}  \ L^\infty(0,T;L^2(\Omega))  \ \text{weakly star},
\end{equation}
\begin{equation}
\label{weak2}
\rho_{\epsilon'}\rightharpoonup\rho  \ \text{in} \ L^2(0,T;H^1(\Omega))  \ \text{weakly}.
\end{equation}
Due to Theorem \eqref{thHgamma} and Aubin-Lion's Lemma, we have
\begin{align}
\label{strong1}
u_{\epsilon'}\to u \  & \text{in}  \  L^2(0,T;L^2(\Omega))  \ \text{strongly}, \\
\label{strong2}
\rho_{\epsilon'}\to \rho \  & \text{in} \ L^2(0,T;L^2(\Omega)) \ \text{strongly}.
\end{align}
In addition, by the boundedness of the elements of the matrix $A_0^{-1}$ and $x'$ we have that 
\begin{align*}
A_0^{-1}(\rho_{\epsilon'})\to A_0^{-1}(\rho)\ \ \ &\text{strongly  in} \ L^p(0,T;L^p(\Omega)),\\
x'(\rho_{\epsilon'})\to x'(\rho) \ \ \ &\text{strongly  in} \ L^p(0,T;L^p(\Omega)),
\end{align*}
for any $p<\infty$. Moreover we also have 
\begin{equation*}
\nabla x'(\rho_{\epsilon'})\ \rightharpoonup \nabla x'(\rho_*) \ \ \text{in} \ L^2(0,T;L^2(\Omega)) \ \ \text{weakly}.
\end{equation*}
Now we are ready to pass to the limit for $\epsilon'\to0$.	\\ \indent
In the weak formulation \eqref{weakaprp} we take a test function $r\in C^\infty_0(\Omega\times[0,T);\mathbb{R})$ of the form $r(x,t)=\tilde{r}(x)\psi(t)$ where $\tilde{r}\in C^\infty_0(\Omega)$ and $\psi\in C^\infty_0([0,T))$. Then we have 
\begin{equation}
\label{lim2}
\int_{0}^{T}\psi(t)\bigg[\epsilon'\frac{d}{dt}\iO p_{\epsilon'}\tilde{r} \ dz +\iO(\div u_{\epsilon'})\tilde{r} \ dz\bigg]dt=0 \ \ \ \text{for  any} \ \psi \in C^\infty_0([0,T)).
\end{equation} 
As $\epsilon'\to 0$, by using \eqref{weak1} and \eqref{weakpres}, from \eqref{lim2}, we get that
\begin{equation*}
\iO (\div u)\tilde{r}dz=0, \  \text{for  any} \ \tilde{r} \ \text{in} \ C^\infty_0(\Omega),
\end{equation*}
which implies,
\begin{equation*}
\div u=0.
\end{equation*}
Hence $u\in \mathcal{V}$. Let us check that $(u,\rho)$ is a weak solution to \eqref{inceq1}-\eqref{inceq4}.\\ \indent

Let $v\in C^\infty_0(\Omega\times[0,T); \mathbb{R}^d)$ with $\div v=0$ as test function in \eqref{weakapru}, we get
\begin{equation}
\label{ueps'}
\begin{split}
&-\int_{0}^{T}\iO u_{\epsilon'}\cdot \dt v \ dtdz+\int_{0}^{T}\iO \nabla u_{\epsilon'}:\nabla v \ dtdz +\int_{0}^{T}\hat{b}(u_{\epsilon'},u_{\epsilon'},v) \ dt\\ &=
\iO u_0\cdot v(\cdot, 0) \ dz
+\int_{0}^{T}\iO f\cdot v \ dtdz, \ \ \text{for  all} \ v \in V.
\end{split}
\end{equation}  
For the densities $\rho_{\epsilon'}$, taking $q \in C^\infty_0(\bar{\Omega}\times[0,T); \mathbb{R}^N)$ with $\nabla q \cdot \nu|_{\partial \Omega}=0$, the weak problem \eqref{defweakrho} is 
\begin{equation}
\label{rhoeps'}
\begin{split}
&-\int_{0}^{T}\int_{\Omega}\rho_{\epsilon'} \cdot \dt q \ dzdt+\int_{0}^{T}\int_{\Omega}\nabla q : A_0^{-1}(\rho_{\epsilon'})\nabla x'(\rho_{\epsilon'}) \ dzdt\\
&+\int_{0}^{T}\hat{b}(u_{\epsilon'},\rho_{\epsilon'},q)\ dt+\frac{1}{2}\int_{0}^{T}\iO\div u_{\epsilon'} \big(\rho_{\epsilon'}\cdot q \big)dz\\
&=\iO(\rho')^0\cdot q(\cdot,0)dz,
\end{split}
\end{equation}

Now we can pass to the limit thanks to the weak convergence results \eqref{weak1}, \eqref{weak2} and the strong one \eqref{strong1}, \eqref{strong2}, and we obtain that the equation \eqref{ueps'}, as $\epsilon'\to 0$, becomes
\begin{equation}
\label{limit1}
\begin{split}
&-\int_{0}^{T}\iO u\cdot \dt v \ dtdz+\int_{0}^{T}\iO \nabla u:\nabla v \ dtdz \\
&+\int_{0}^{T}\hat{b}(u,u,v) \ dt\\ &=
\iO u_0\cdot v(\cdot, 0) \ dz
+\int_{0}^{T}\iO f\cdot v \ dtdz, \ \ \text{for  all} \ v \in V.
\end{split}
\end{equation}
While for the densities, taking the limit as $\epsilon'\to 0$, in the equation \eqref{rhoeps'},  we get
\begin{equation}
\label{limit2}
\begin{split}
&-\int_{0}^{T}\int_{\Omega}\rho \cdot \dt q \ dzdt+\int_{0}^{T}\int_{\Omega}\nabla q : A_0^{-1}(\rho)\nabla x'(\rho) \ dzdt\\
&+\int_{0}^{T}\hat{b}(u,\rho,q)\ dt+\frac{1}{2}\int_{0}^{T}\iO\div u \big(\rho\cdot q \big)dz\\
&=\iO(\rho^0)'\cdot q(\cdot,0)dz.
\end{split}
\end{equation}
By the definition of the operator $\hat{b}$ given in \eqref{defbhat}, since $\div u=0$, we observe that \eqref{limit1}  and \eqref{limit2} are exactly the weak formulations for the system \eqref{inceq1}-\eqref{inceq4}. \\ \indent
Finally to prove the result for the gradient of the pressure, notice that if we test \eqref{weakapru} against a function $v$ such that $\div v \neq 0$, we have that
\begin{align*}
\int_{0}^{T}\iO\nabla p_\epsilon\cdot v=&\int_{0}^{T}\iO \ue \cdot \dt v \ dzdt+\iO u_0\cdot v(\cdot,0) dz-\int_{0}^{T} \hat{b}(\ue,\ue,v) dt\\&-\int_{0}^{T}\iO \nabla u_\epsilon:\nabla v \ dzdt +\int_{0}^{T}\iO f\cdot v \ dzdt,
\end{align*}
and as $\epsilon\to 0$ the right hand side converges in $H^{-1}(\Omega)$ to a function that, comparing with \eqref{inceq3} is $\nabla p$.

\end{document}